\documentclass[12pt,draftcls,onecolumn]{IEEEtran}

\usepackage{amssymb}
\usepackage{amsmath,amsthm,enumerate,pdfsync,asymptote}  
\usepackage{setspace} 
\usepackage{graphicx,subfig}
\usepackage{amssymb}
\usepackage{setspace}
\usepackage{graphics,framed}
\usepackage{pgf}
\usepackage{tikz}
\usepackage{pgfplots}

\newtheorem{Definition}{Definition}
\newtheorem{Example}{Example}
\newtheorem{Proposition}{Proposition}
\newtheorem{Theorem}{Theorem}
\newtheorem{Lemma}{Lemma}
\newtheorem{Corollary}{Corollary}	
\newtheorem{Remark}{Remark}
\newcommand{\R}{\mathbb{R}}

\newcommand{\scone}[3]{%
\begin{scope}[rotate=#3,xshift=#1,yshift=#2]
\def\mypath{ (.-.2,0) -- +(-.2,.8) arc (180:0:.8) -- +(-.2,-.8) arc (-180:0:.2) } 
\fill [gray] mypath;
}
\usetikzlibrary{decorations.markings}

\tikzset{%
  >=latex, 
  inner sep=0pt,%
  outer sep=2pt,%
  mark coordinate/.style={inner sep=0pt,outer sep=0pt,minimum size=3pt,
    fill=black,circle}%
}

\makeatletter
\newsavebox{\sfe@box}
{\color@endgroup\egroup\subfloat[\sfe@caption]%
{\usebox{\sfe@box}}}
\makeatother

\usepackage{mdwmath}
\usepackage{mdwtab}

\usepackage{stfloats}

\begin{document}

\title{Singularities and global stability of decentralized formations in the plane}

\author{M.-A. Belabbas\thanks{33 Oxford St, Cambridge MA 02138}
}

\maketitle
\markboth{}%
{Shell \MakeLowercase{\textit{et al.}}}

\begin{abstract}
Formation control is concerned with the design of control laws that stabilize agents at given distances from each other, with the constraint that an agent's dynamics can depend only on a subset of other agents. When the information flow graph of the system, which encodes this dependency, is acyclic, simple control laws are known to globally stabilize the system, save for a set of measure zero of initial conditions.  The situation has proven to be more complex when the graph contains cycles; in fact,  with the exception of the cyclic formation with three agents, which is stabilized with laws similar to the ones of the acyclic case, very little is known about formations with cycles.  Moreover, all of the control laws used in the acyclic case  fail at stabilizing more complex cyclic formations. In this paper, we explain why this is the case and show that a large class of planar formations with cycles cannot be globally stabilized, even up to sets of measure zero of initial conditions. The approach rests on  relating the information flow  to  singularities in the dynamics of formations. These singularities are in turn shown to make the existence of  stable configurations that do not satisfy the prescribed edge lengths generic.

\end{abstract}

\section{Introduction}

Consider the following problem, depicted in Figure~\ref{fig:triang1}. Three autonomous agents with positions $x_1,x_2$ and $x_3$  evolve in the plane according to first order dynamics, agent $1$ observes the position of agent $2$, agent $2$ the position of agent $3$ and agent $3$ the position of agent $1$. Can the agents stabilize at prescribed distances $d_1, d_2, d_3$ from each other?

Problems of this type, which fall under the broader class of decentralized control problems, have been a focal point of attention of control theory for the past decade or more, as they arise in a wide variety of natural (think schooling, herding, etc.) and engineering situations (think autonomous vehicles or decentralized power systems).  The vast majority of control laws proposed in these contexts are so-called \emph{gradient control laws}, named after the fact that every agent tries to minimize its own objective function. While such control laws work well in centralized systems with a small number of agents, they become ineffectual at stabilizing a given configuration when this number increases or when the system is decentralized due to the appearance of a very large number of stable configurations that do not respect the desired inter-agent distances. 
The objective of this paper is to address the effect of decentralization on the appearance of stable, undesired equilibrium configurations.  In order to do so, we will consider a broader class of control laws than gradient  laws and analyze a particular four-agents formation, called the 2-cycles. This four-agents formation  was exhibited in~\cite{cao2010festschrift} to illustrate the shortcomings of the current methods in formation control and decentralized systems, as it resisted attempts to either define globally stabilizing control laws or prove their non-existence.

We will use ideas from singularity and bifurcation theory~\cite{arnold_bifurcation,golub_stewart}  to show that the 2-cycles, and trivially systems containing it as a subformation, are not globally stabilizable. Singularities have not often appeared in the study of global control design, the reason behind this fact is that they are, in general, easily avoided by considering a small perturbation of the system and can thus be made irrelevant to the dynamics. By opposition, we will show here that \emph{the information flow constraints inherent to decentralized control can make such singularities unavoidable}.

\begin{figure}[ht]
\begin{center}
\subfloat[]{\label{fig:triang1}
\begin{tikzpicture}[scale = .5] 

\node [fill=black,circle, inner sep=1pt,label=90:$x_1$] (1) at ( 0, 0) {};

\node [fill=black,circle, inner sep=1pt,label=-135:$x_2$] (2) at (1.8 ,-1) {};
\node [fill=black,circle, inner sep=1pt,label=-45:$x_3$] (3) at (1.5 ,1.1) {};

\draw [-stealth ] (1) -- (2);
\draw [-stealth] (2) -- (3);
\draw [-stealth] (3) -- (1);

\end{tikzpicture}}\qquad
\subfloat[]{\label{fig:s2c}
\begin{tikzpicture}[scale = .5] 

\node [fill=black,circle, inner sep=1pt,label=180:$x_1$] (1) at ( 0, 0) {};
\node [fill=black,circle, inner sep=1pt,label=135:$x_2$] (2) at (1.5 ,1.5) {};
  \node [fill=black,circle, inner sep=1pt,label=-45:$x_3$] (3) at (3.5 , .5) {};
\node [fill=black,circle, inner sep=1pt,label=-45:$x_4$] (4) at ( 2,-.5) {};

\draw [-stealth,  ] (1) -- (2);
\draw [-stealth, ] (1) -- (4);
\draw [-stealth, ] (3) -- (1);
\draw [-stealth, ] (2) -- (3);
\draw [-stealth, ] (4) -- (3);
\end{tikzpicture}
}
\end{center}

\caption{\small (a). Three agent in a cyclic formation in the plane. Agent $1$ observes agent $2$, which observes agent $3$ which in turn observes agent $1$. An arrow pointing from agent $i$ to agent $j$ thus implies that the dynamics of agent $i$ is allowed to depend to the state of agent $j$.  (b). The 2-cycles formation.}\label{fig:3out}
\vspace{-.5cm}
\end{figure}
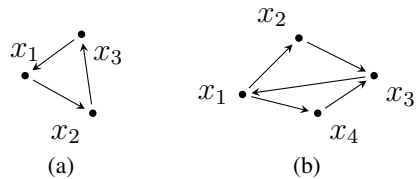

Going back to the example of  three agents in the plane, there are at least two ways in which such systems arise in practical situations:
\begin{itemize} 
\item  Every agent chooses independently the distance $d_i$ at which to stabilize from another agent. Can we design control laws so that when the agents evolve, they will all stabilize at their desired distances from almost all initial positions?

\item A central authority decides on a target configuration and sends each agent only part of the complete description of the target, e.g. by sending them only the inter-agent distance they have to satisfy. Can we design a control law that will stabilize every target configuration?

\end{itemize}

These two points of view of course lead to the same mathematical problem, though they arise in different contexts. In the first case, the question may arise in the study of natural flocks or non-cooperative settings, where agents are incapable or not wanting to communicate their objectives. In the second case, scenarios involving UAV's which are designed to achieve a global objective, but which for secrecy or efficiency reasons are only given part of the global objective, can give rise to such questions. For a more detailed discussion of local and global objectives of a decentralized system, see~\cite{belabbas_icca_knowns}.

We now summarize the extant relevant work and methods used, see~\cite{belabbas_icca_mathformation} for more details. We first mention that most issues in formation control arise when the information flow graph of the system contains cycles. Indeed, a cycle-free formation, which necessarily has a leader (i.e. an agent that does not observe any other agents) can  easily be handled via the use of different time-scales depending on how  distant from a leader an agent is. See~\cite{cao2008acyclic} for a  detailed sketch of how a gradient-based control law globally stabilizes, save for a set of initial conditions of measure zero, acyclic formations. For related wor, dealing with undirected formations, we refer to~\cite{krick08} and references therein.

\begin{Remark}We mention here that the existence of a set of initial conditions that will not lead to the desired configuration stems from the nontrivial topology of the state-space of formation control, essentially the topology of a complex projective space~\cite{belabbasSICOpart1}. We say that a control law \emph{almost surely stabilizes} an equilibrium if the system stabilizes that equilibrium from almost all initial conditions. We come back to this in Section~\ref{sec:stability}.
\end{Remark}

Local stabilization of formations that contain cycles has similarly been  investigated in~\cite{yu09}, where it is shown that the widely-used gradient law can be modified by adjusting some gains to stabilize any desired target configuration. One caveat to the result is that the gains are not evaluated locally by the agents, resulting in a centralized "design phase" followed by a decentralized "implementation phase" for the agents. It is shown in~\cite{belabbasSICOdecentralized} that the gains cannot be locally evaluated by the agents in the case of the 2-cycles formation, i.e. that a decentralized design phase is not possible.

This leaves us with the case of global stabilization of formations that contains cycles. The extant work consists of  the thorough analysis of the triangle formation, or 1-cycle, done in~\cite{cao07cdc} and related publications, where it is shown that the gradient control law almost surely globally stabilizes almost any configurations. We show in this paper that the second simplest formation with cycles, i.e. the 2-cycles formation,  cannot be globally stabilized---even modulo sets of measure zero of initial conditions---by a broad class of control laws,  including the control laws used in prior work.

In order to prove the main result, we introduce two definitions. First, the already mentioned \emph{almost sure stability}, in Section~\ref{sec:stability}, which is needed to formalize the idea of global stability modulo sets of measure zero which has appeared in prior work on formation control. The second is the one of robustness for nonlinear systems; in a few words, we say that a control law is robust  if its effect (e.g. stabilization) persists after small perturbations of the dynamics. This is closely related to structural stability of systems and requires some use of transversality conditions, discussed in the Appendix, to make rigorous. 

We conclude this introduction by an example: we illustrate how the gradient-based control law  used in many works on directed formation control~\cite{krick08,yu09,cao07cdc} fails to globally stabilize the 2-cycles formation by stabilizing around an undesired configuration.

Let $x_i \in \R^2$ be the positions of the agents and $d_i$ be positive numbers, the target edge lengths. The decentralized system is explicitly given by:
\begin{equation}\label{eq:symEx}
\left\lbrace \begin{array}{rcl}
\dot x_1 &=& (\|x_1-x_2\|-d_1)(x_1-x_2) +(\|x_1-x_4\|-d_5)(x_1-x_4) \\
\dot x_2 &=& (\|x_2-x_3\|-d_2)(x_2-x_3)\\
\dot x_3 &=& (\|x_3-x_1\|-d_3)(x_3-x_1)\\
\dot x_4 &=& (\|x_4-x_3\|-d_4)(x_4-x_3)
\end{array}\right.
\end{equation}

We show in Figure~\ref{fig:simuxx} the results of simulations  for the vector of target distances $d_1=2.0, d_2= 2.6, d_3= 2.0, d_4=1.4, d_5= 3.3$ and illustrate the appearance of a stable undesired configuration, which is accompanied by an unstable desired configuration.

\begin{figure}[ht]
\begin{center}
\subfloat[Configuration $D1$]{
\includegraphics[width=.25\columnwidth]{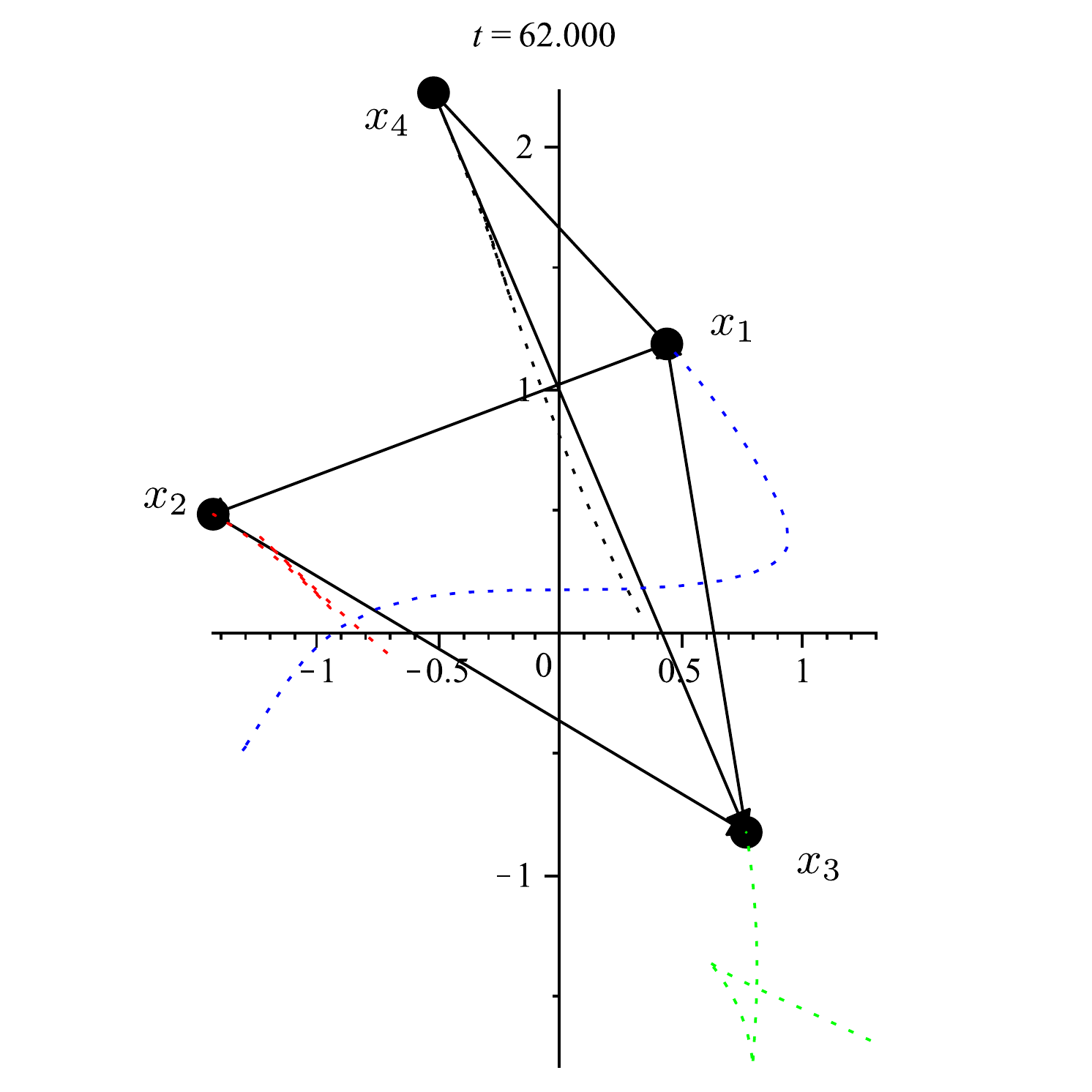}\label{fig:sim1d1}
}
\subfloat[Configuration $D2$]{
\includegraphics[width=.25\columnwidth]{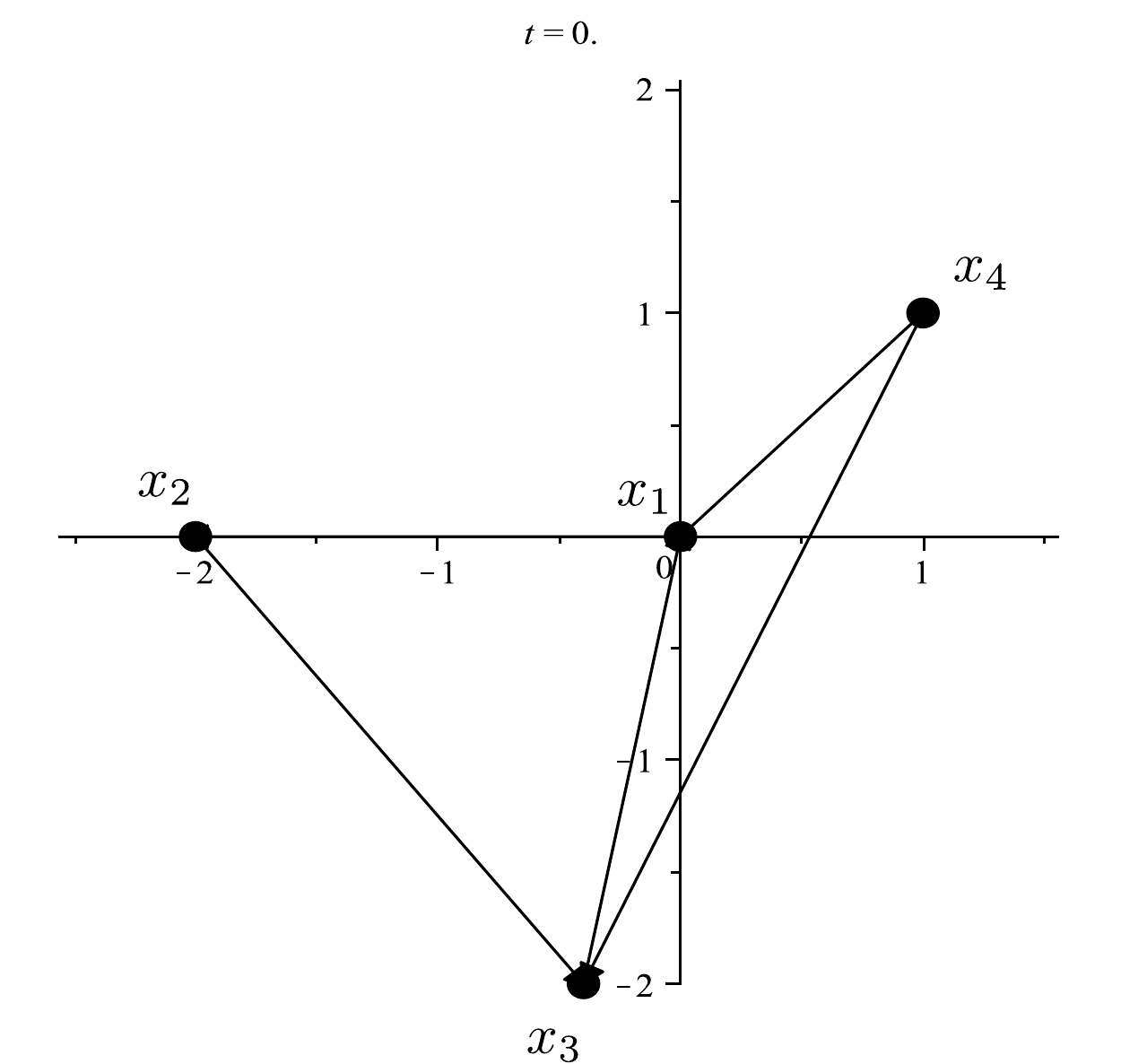}\label{fig:sim1d2}
}
\subfloat[Configuration $A1$]{
\includegraphics[width=.25\columnwidth]{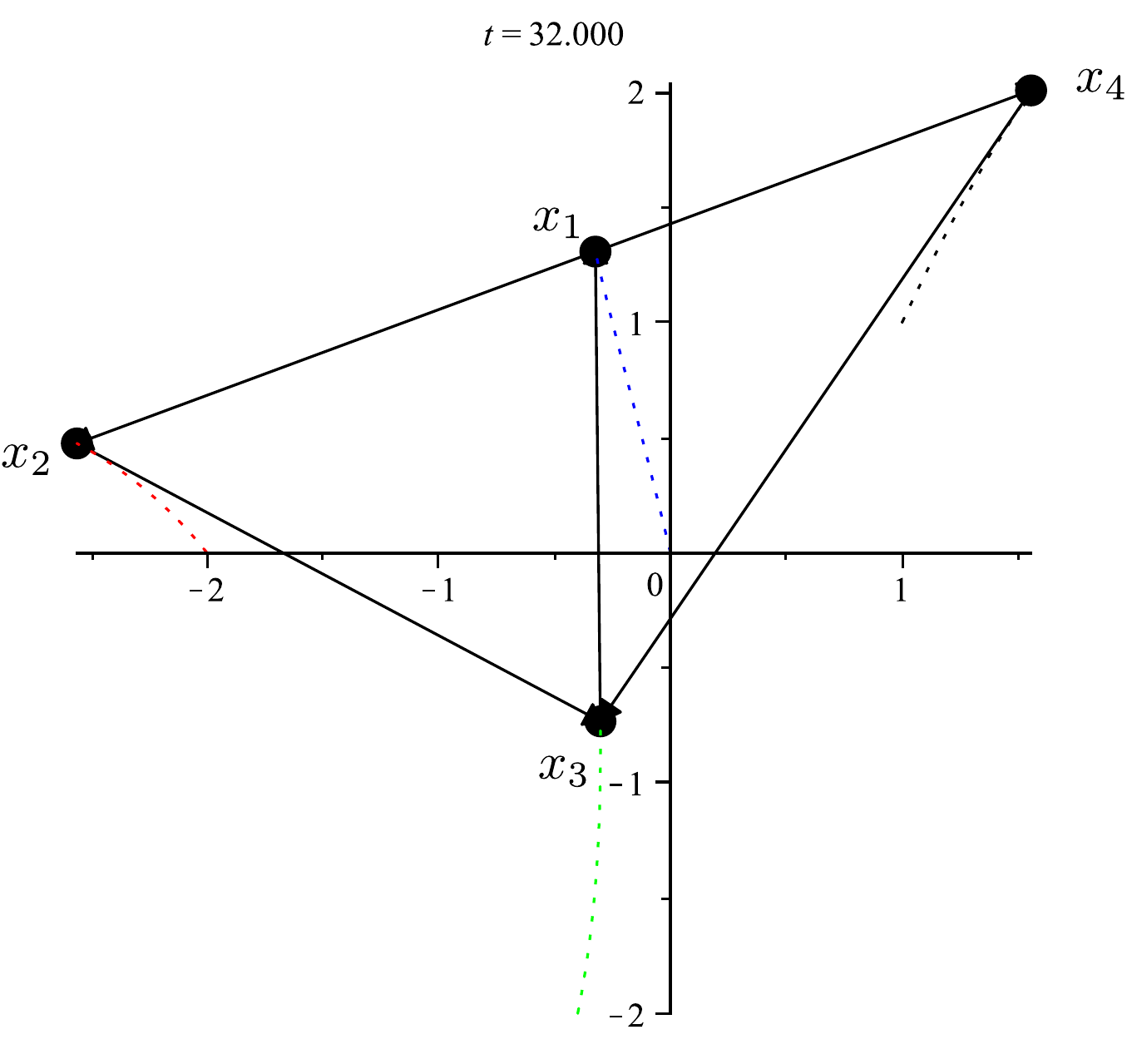}\label{fig:sim1a1}
}
\caption{Simulation results for the decentralized system of Equation~\eqref{eq:symEx} with a $d \in \mathcal{L}_c$. The dotted lines represent the trajectories followed by the agents. Both configurations $D_1$ and $D_2$ represent design target frameworks. $D_1$ is locally stable whereas $D_2$ is locally unstable.  The configuration $A_1$ is also an equilibrium of the dynamics, though one for which the target distances are not respected: we call it \emph{ancillary equilibrium}.
 A linearization of the system gives that the spectra of the Jacobians  are given by $(-17.5 + 1.3i, -17.5 - 1.3i, -11.9, -7.9, -0.6)$, $( 0.6,-18.6 + 3i, -18.6- 3i, -9.4 + 3.1i, -9.4 - 3.1i)$ and $( -23.4 + 4.8i, -23.4 - 4.8i, -11 + 2.8i, -11 - 2.8i,-1.6  )$ for the configurations $D1$, $D2$, and $A1$ respectively. Hence $A1$ is locally exponentially stable and thus there is an open set of initial conditions that lead to an ancillary configuration. }
\label{fig:simuxx}
\end{center}
\end{figure}

\section{Almost sure stability}\label{sec:stability}

Consider the control system \begin{equation}
\label{eq:sysu1}\dot x = f(x,u(x))
\end{equation} where $x \in {M}$, a smooth manifold, and all functions are assumed smooth. To justify the definition of almost sure stability, we restrict ourselves to the case of hyperbolic dynamics, i.e. having the property that the linearization of the dynamics at a zero has eigenvalues with non-zero real-parts.

We are interested in  \emph{global} results about stabilization of an equilibrium configuration. From~\cite{belabbasSICOpart1}, we know that because of their invariance to rigid transformations of the plane, formation control systems evolve on the manifold $M= \mathbb{C}P(n-2) \times \R$, which is of dimension $2n-3$. The well-known Poincar\'e-Hopf equality, which we illustrate in the case of the circle in Figure~\ref{fig:poincare}, relates the \emph{index} of isolated zeros of differentiable vector fields on $M$ to a global topological characteristic of $M$, its Euler characteristic~\cite{guck}.

On the one-hand, the Euler characteristic of $M$ is known to be $n-1$. For $x_0$ an isolated zero of a vector field, the useful formula $$\operatorname{ind}(x_0) = \operatorname{sign}\left(\det(\frac{\partial f}{\partial x}|_{x_0})\right)$$ tells us that  a stable  equilibrium has an index of $-1$ (since the dimension of $M$ is odd). Putting these two simple observations together, we conclude that global stabilization of an equilibrium is not possible for formation control, since stabilizing a zero will force the appearance  of other zeroes $a_i$ to satisfy the Poincar\'e-Hopf equality: $$\operatorname{ind}(x_0) + \sum_i \operatorname{ind}(a_i)= n-1 \Rightarrow \sum_i \operatorname{ind}(a_i) = n.$$

\begin{figure}
\begin{center}
\begin{tikzpicture}[decoration={
    markings,
    mark=at position 0.5 with {\arrow{>}}}
    ] 
\draw[postaction={decorate}] (-3.5,1) arc (90:-90:1cm);
\draw[postaction={decorate}] (-3.5,-1) arc (-90:-270:1cm);


\draw[postaction={decorate}] (1,0) arc (0:-90:1cm);

\draw[postaction={decorate}] (-0,1) arc (90:-0:1cm);
\draw[postaction={decorate}] (-0,1) arc (90:180:1cm);
\draw[postaction={decorate}] (-1,0) arc (180:270:1cm);
\node [fill=black,circle, inner sep=1pt,label=90:$1$] (1) at ( 0, 1) {};
\node [fill=black,circle, inner sep=1pt,label=-90:$-1$] (1) at ( 0, -1) {};


\node [fill=black,circle, inner sep=1pt,label=90:$-1$] (1) at ( 3.5, 1) {};
\node [fill=black,circle, inner sep=1pt,label=0:$1$] (1) at ( 4.5, 0) {};
\node [fill=black,circle, inner sep=1pt,label=-90:$-1$] (1) at ( 3.5, -1) {};
\node [fill=black,circle, inner sep=1pt,label=180:$1$] (1) at ( 2.5, 0) {};
\draw[postaction={decorate}] (4.5,0) arc (0:90:1cm);
\draw[postaction={decorate}] (4.5,0) arc (0:-90:1cm);
\draw[postaction={decorate}] (2.5,0) arc (180:90:1cm);
\draw[postaction={decorate}] (2.5,0) arc (180:270:1cm);
\end{tikzpicture}

\end{center}
\caption{We represent three continuous vector fields on the circle. The one on the left has no equilibrium, the one in the center has one stable equilibrium and one unstable equilibrium and the one on the right has two stable equilibria and two unstable equilibria; the indices of the equilibria are indicated on the figure. Because the sum of the indices is constrained to be zero, there is an even number of equilibria and, in particular, no continuous system on the circle can be globally stable.}\label{fig:poincare}
\end{figure}
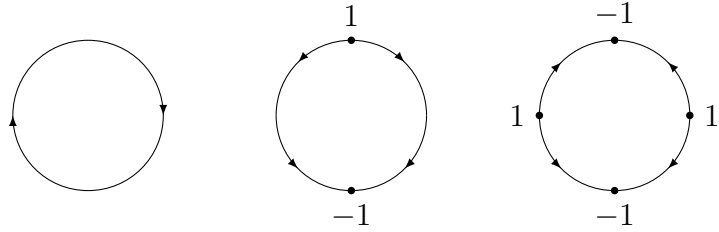

On the other hand, we know from the Hartman-Grobman theorem that the dimensions of the stable and unstable manifolds of an equilibrium  are given by the numbers of eigenvalues of the Jacobian with negative and positive real parts respectively.  The attraction basin of an equilibrium is thus of codimension at least one (in other words, it is a thin set) unless all its eigenvalues have negative real parts.

Hence, from a practical standpoint, if we  only require that the control $u(x)$  makes  one equilibrium stable and all other equilibria either saddles or unstable, we obtain a  system that behaves essentially like a  globally stable system---since a vanishingly small perturbation would ensure that the system, if at a saddle or unstable equilibrium,  evolves to the unique stable equilibrium---while leaving more room to possibly satisfy  the Poincar\'e-Hopf equality. We formalize and elaborate on this observation here. 

Let $\mathcal{E}_d$ be a finite subset of $M$ containing  configurations that we would like to stabilize via feedback.  All configurations in $\mathcal{E}_d$ are equally appropriate for the stabilization purpose. We are thus interested in the design of a smooth feedback control $u(x)$ that will stabilize the system to any point $x_0\in \mathcal{E}_d$. We call these points the \emph{design targets} or \emph{design equilibria}:

$$\mathcal{E}_d = \lbrace x_0 \in M \mbox{ s.t. } x_0 \mbox{ is a design equilibrium} \rbrace $$

Let $$\mathcal{E}  = \lbrace x_0 \in M \mbox{ s.t. } f(x_0,u(x_0)) = 0 \rbrace, $$ the set of equilibria of~\eqref{eq:sysu1}. We  assume that $\mathcal{E}$ is {finite}. 

As explained above, when the system evolves on a non-trivial manifold, the Poincar\'e-Hopf equality, or the more refined Morse inequalities~\cite{smale1967}, make it unreasonable  to expect that there exists a control $u(x)$ that makes the design equilibria the \emph{only} equilibria of the system, i.e. such that $\mathcal{E}_d= \mathcal{E}$.  We call the additional equilibria, that are introduced by the non-trivial topology of the space, \emph{ancillary equilibria}: $$\mathcal{E}_a = \mathcal{E}-\mathcal{E}_d.$$

\begin{figure}
\begin{center}
\subfloat[]{
\begin{tikzpicture} 
\begin{scope}
\def\R{1.5} 
\def\angEl{35} 
\filldraw[ball color=white] (0,0) circle (\R);

\coordinate[draw, shape = circle, fill, inner sep =  1pt] (D1) at (.4*\R, .7*\R);
\node[below=3pt] at (D1) {\tiny $D_1$};
\node at (-1.24,1.24) {$M$};

\coordinate[draw, shape = circle, fill, inner sep =  1pt] (D2) at (-.4*\R, -.7*\R);
\node[above=3pt] at (D2) {\tiny $D_2$};
\end{scope}
\end{tikzpicture}
} \qquad\quad
\subfloat[]{
\begin{tikzpicture}
\begin{scope}[xshift = 0cm]
\node at (-1.24,1.24) {$M$};

\def\R{1.5} 
\def\angEl{35} 
\filldraw[ball color=white] (0,0) circle (\R);

\coordinate[draw, shape = circle, fill, inner sep =  1pt] (D1) at (.4*\R, .7*\R);
\node[below=3pt] at (D1) {\tiny $D_1$};
\coordinate[draw, shape = circle, fill, inner sep =  1pt ] (D2) at (-.4*\R, -.7*\R);
\node[above=3pt] at (D2) {\tiny $D_2$};

\coordinate[draw, shape = circle, , inner sep =  1pt, red] (A1) at (.3*\R,- .7*\R);
\node[above=3pt] at (A1) {\tiny $A_1$};
\coordinate[draw, shape = circle, , inner sep =  1pt, red] (A2) at (-.3*\R, .7*\R);
\node[below=3pt] at (A2) {\tiny $A_2$};

\end{scope}
\end{tikzpicture}

}

\caption{Consider a control system defined on the sphere $M$. Assume that we want to design a control law such that the system stabilizes almost surely to either $D_1$ or $D_2$ as depicted in (a) above. We call  $D_1$ and $D_2$ design equilibria and write $\mathcal{E}_d = \lbrace D_1, D_2 \rbrace$.  A continuous feedback control law $u(x)$ on the sphere that has either $D_1$ or $D_2$, or both, as zeros may introduce additional zeros, which we call ancillary equilibria. We assume that there are two such equilibria $A_1$ and $A_2$ and write $\mathcal{E}_a = \lbrace A_1, A_2\rbrace. $ The control law $u(x)$ makes the system almost surely stable if {\it at least one element} of $\mathcal{E}_d$ is stable and {\it no element} in $\mathcal{E}_a$ is stable. }
\end{center}
\end{figure}
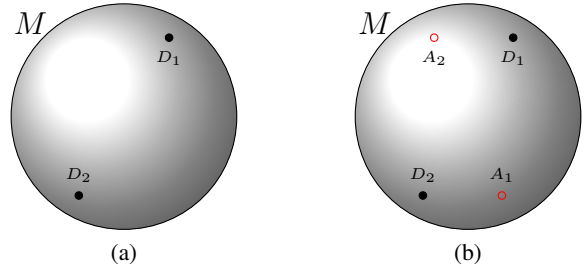

 We decompose the set $\mathcal{E}$ into \emph{stable} equilibria, by which we mean  equilibria such that \emph{all the eigenvalues} of the linearized system have a negative real part, and \emph{unstable equilibria}, where \emph{at least one eigenvalue} of the linearization has a positive real part. Observe that under this definition,   saddle points are considered unstable. 
In summary: $$\mathcal{E} = \mathcal{E}_s \cup \mathcal{E}_u$$ where
$$\mathcal{E}_s = \lbrace x_0 \in \mathcal{E}\ |\ x_0 \mbox{ is stable} \rbrace \mbox{ and } \mathcal{E}_u = \lbrace x_0 \in \mathcal{E}\  |\ x_0 \mbox{ is unstable} \rbrace.$$

With these notions in mind, we introduce the following definition:

\begin{Definition}
Consider the smooth control system $\dot x = f(x,u(x))$ where $x \in M$ and the set $\mathcal{E}$ of equilibria of the system is finite. Let $\mathcal{E}_d \subset M$ be a finite set. We say that $\mathcal{E}_d$ is
\begin{enumerate}
\item \emph{feasible} if we can choose a smooth $u(x)$ such that $\mathcal{E}_d \cap \mathcal{E} \neq \varnothing$.
\item \emph{almost-surely stabilizable} if we can choose a smooth $u(x)$ such that $\mathcal{E}_s \subset \mathcal{E}_d$.

\end{enumerate}
When the set $\mathcal{E}_d$ is clear from the context, we say that the system is feasible or almost-surely stable.\footnote{In the earlier publication~\cite{belabbascdc11sub2}, we termed almost-sure stability type-A stability}
\end{Definition}

The set $\mathcal{E}_d$ is feasible if we can choose $u(x)$ such that \emph{at least one} equilibrium  of the system is a design target. It is said to be \emph{almost surely} stable if the system  stabilizes to $\mathcal{E}_d$ almost surely for all initial conditions on $M$.
The usual notion of global stability is a particular instance of almost-sure stability; indeed, it  corresponds to having $u(x)$ such that  $\mathcal{E}_d=\mathcal{E}=\mathcal{E}_s$. 

Looking at  the contrapositive of this definition,  a system  is \emph{not almost-surely stable} if there exists a set of initial conditions of codimension zero that leads to an ancillary equilibrium.

\begin{Example}
Consider a system $$\dot x = x(1-kx^2)$$ where $k \in \R$ is a feedback parameter to be chosen by the user. We show that any $\mathcal{E}_d \subset (0,\infty)$ is not almost-surely stable. We first observe that the system has an equilibrium at $0$ and two equilibria at $x = \pm \sqrt{1/k}$ if $k >0$.  The system is thus feasible for any $\mathcal{E}_d \subset \R$.  The Jacobian of the system is $1$ at $x=0$ and $-2$ at $x=\pm \sqrt{1/k}$. For $k>0$, the above says that $$\mathcal{E}= \lbrace 0, \pm \sqrt{1/k} \rbrace= \underbrace{\lbrace \sqrt{1/k} \rbrace}_{\mathcal{E}_d} \cup \underbrace{ \lbrace 0, -\sqrt{1/k} \rbrace}_{\mathcal{E}_a}.$$  From the linearization of the system, we have that $$\mathcal{E}_s = \lbrace \pm \sqrt{1/k}\rbrace \mbox{ and }\mathcal{E}_u = \lbrace 0 \rbrace.$$ We conclude that $\mathcal{E}_s \nsubseteq \mathcal{E}_d$ and the system is not almost-surely stable. Indeed, all initial conditions $x_0 < 0$ result in the system stabilizing at an ancillary equilibrium.
\end{Example}

\section{Genericity and robustness}\label{sec:sing}

We now move on to the first of the two main technical ingredients necessary for the proof of the results below. The second ingredient, singularities and bifurcations, is presented in the next section.

Let $\mathcal P$ be a binary-valued function on a topological space $S$, indicating whether a given property is satisfied. In more detail, if $u \in S$, we say that $u$ satisfies $\mathcal P$ is $\mathcal P(u) =1$. We have the following definition:

\begin{Definition}[Robustness] An element $u$ of a topological space $S$ satisfies the  property $\mathcal{P}$  \emph{robustly} if for all $u'$  in a neighborhood of $u$ in $S$ we have $\mathcal P(u')=1 $. A property $\mathcal P$ is robust if there exists an open set $U \subset S$ such that $\mathcal P(U) =1 $. 
\end{Definition}

The property we will be dealing with here is stability: we want to find a $u$ that stabilizes a  system around an equilibrium, and desire the stabilization to be robust. In practical terms, if a property satisfied only at \emph{non-robust} $u$'s, then it fails to be satisfied under the slightest error in modelling or measurement. 

Related to robustness is the notion of genericity: a property $\mathcal{P}$ is \emph{generic} for a topological space $S$ if it is true on an everywhere dense intersection of open sets of $S$.

Everywhere dense intersections of open sets are sometimes  called \emph{residual} sets~\cite{arnold_bifurcation}.

\begin{Remark}
We emphasize that when we seek a robust control law $u(x)$ for stabilization, we seek a control law such that  the equilibrium that is to be stabilized remains stable under small perturbations in $u(x)$. The equilibrium, however, may move in the state space. For example, assume that the system $$\dot x = f(x,u(x))$$ has the origin as a stable equilibrium. If for all $\tilde u(x)$ in an appropriate set of perturbations, the system $$\dot x=f(x,u(x) + \varepsilon \tilde u(x))$$ has a stable equilibrium at a point $z(\varepsilon )$ near the origin, then the control law $u(x)$ is robust. If, on the contrary, the equilibrium disappears or becomes unstable, then $u(x)$ is not robust.
\end{Remark}

If $1-\mathcal P$, the negation of $\mathcal P$, is generic, then there is no robust $u$ that satisfies $\mathcal P$. Indeed, if $1- \mathcal P$ is generic, then $\mathcal{P}$ is verified on at most a nowhere dense closed set. In particular, $\mathcal P$ is not verified on an open set. The main tool to handle genericity are jet spaces and Thom transversality theorem. We will use the results in some parts below and refer the reader to the appendix for more information.

\section{Singularities, transcritical bifurcation and the logistic equation}\label{sec:logistic}

We recall a few definitions from dynamical systems theory. Consider a dynamical system of the form\begin{equation}
\label{eq:defeqm}
\dot x = f_\mu(x)\end{equation}where $x \in {M} $, an $n-$dimensional manifold, and $\mu\in \R^k$ is a vector of parameters on which the system smoothly depends.

\begin{Definition}[Hyperbolic and singular equilibria and bifurcation value]\hskip 0pt
\begin{enumerate} \item An equilibrium $x_0$ is called \emph{hyperbolic} if the eigenvalues of the linearization at $x_0$ have  non-zero real-parts. It is called \emph{singular} or \emph{degenerate} otherwise.
\item  A value $\mu_0$ in the parameter space $\R^k$ for which the flow of~\eqref{eq:defeqm} has a singular equilibrium  is called a \emph{bifurcation value}.
\end{enumerate}
\end{Definition}

\subsection{The logisitc equation}

The logistic equation, which is  often used to describe systems in which two competing effects---such as supply and demand   or predator and prey--- are at play, is the one-dimensional ODE given by \begin{equation}
\label{eq:logeq}
\dot x = x(\mu-x).
\end{equation} This equation displays what is called a \emph{transcritical} or \emph{transfer of stability} bifurcation at $\mu=0$, which we explain here. Observe that the system has two equilibria, one at $x=0$ and one at $x=\mu$, which coalesce when $\mu=0$. The linearization of the system about $x$ is $$\frac{\partial f}{\partial x}= (\mu-x)-x=\mu-2x.$$ From this linearization, we see that for $\mu >0$, the equilibrium $ x=0$ is unstable whereas the equilibrium $x=\mu$ is stable. The situation is reversed for $\mu<0$.  We conclude that at the bifurcation value $\mu=0$, the two equilibria coalesce and \emph{exchange their stability properties}. We depict the above in Figure~\ref{fig:logeq}. This figure is to be compared to Figure~\ref{fig:log4eq}.

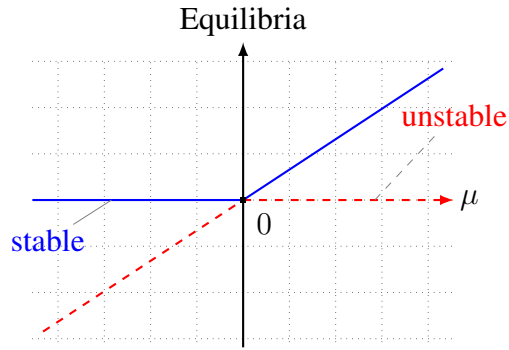
\begin{figure}[]
\begin{center}
\begin{tikzpicture}[scale=.7]
\draw[dotted, step=25pt,very thin](-4,-2.7) grid (4,2.7);

\draw[thick, blue] (-4,0) -- (0,0) node (-2.5,0) [pin=240:stable]{};
\draw[dashed, thick, red,->] (0,0) -- ( 4,0); \node [black] at  (4.3,0) {$\mu$};

\draw[thick, black,->] (0,-2.8) -- (0,3) ;\node  at (0,3.4) {Equilibria};

\draw[dashed, thick, red] (-3.8,-2.5) -- (0,0) node (2.5,0) [pin=60:unstable]{};
\draw[thick, blue] (0,0) -- ( 3.8,2.5) ;
\node [fill=black,inner sep=1pt,label=-45:$0$] at (0,0) {}; 

\end{tikzpicture}
\caption{\small The logistic equation undergoes a transcritical bifurcation when $\mu=0$. The equilibrium $x=0$ is stable for $\mu<0$ and unstable for $\mu>0$.}\label{fig:logeq}
\vspace{-.5cm}
\end{center}

\end{figure}

We   show below  that   the 2-cycles behaves similarly to the logistic equation in the sense that they both exhibit the same type of singularities or bifurcation. The most common approach used to gain some understanding about the behavior of a dynamical system near a singularity relies on the use of the  \emph{center manifold theorem}~\cite{guck}. This theorem establishes the existence of a nonlinear change of coordinates, valid near the equilibrium, where the dynamics can be put in a so-called normal form which is more amenable to analysis. The logistic equation as given in Equation ~\ref{eq:logeq} is such a normal form. This approach is without much hope for success for our purpose unless the control law $u$ is fixed. Indeed, the change of variables involved in the analysis   depends on the  control $u$, and tracking the effect of this dependence through the whole procedure is not feasible for broad classes of control laws. 

In order to sidestep this difficulty, we have recourse to the following result of Sotomayor~\cite{sotomayor73}, which characterizes the generic behavior of dynamical systems near non-hyperbolic fixed-points \emph{without recourse} to the center manifold.

First, recall that for $f:\R^2 \rightarrow \R^2$  a twice differentiable function, its Jacobian is defined as $$\frac{\partial f}{\partial x} = \left[\begin{array}{cc} \frac{\partial f_1}{\partial x} & \frac{\partial f_1}{\partial y}\\ \frac{\partial f_2}{\partial x} & \frac{\partial f_2}{\partial y} \end{array}\right] .$$
Assuming that the Jacobian  has a zero eigenvalue, we denote by $v$ and $w$   corresponding right and left eigenvectors. The Hessian of $f$ is a $2 \times 2 \times 2$ tensor with entries $$\left(\frac{\partial f}{\partial x^2}\right)_{ijk} = \frac{\partial f_i}{\partial x_j\partial x_k}.$$ Hence $$w^T\frac{\partial^2 f}{\partial x^2}(v,v) = \sum_{ijk} w_i \frac{\partial f_i}{\partial x_j\partial x_k} v_jv_k.$$ We have the following theorem:

\begin{Theorem}[Sotomayor]\label{th:soto}
Let $\dot x = f_\mu(x)$ be an ODE in $\R^n$ depending on a scalar parameter $\mu$, with $f$ twice differentiable in $x$ and $\mu$. For $\mu=\mu_0$, assume that the system has an equilibrium $x_0$ satisfying the following conditions:
\begin{enumerate}
\item $\frac{\partial f_{\mu_0}}{\partial x}|_{x_0}$ has a unique zero eigenvalue with left and right eigenvectors $w$ and $v$ respectively. The other eigenvalues are negative.
\item $w^T\frac{\partial f_{\mu}}{\partial \mu}|_{x_0,\mu_0}v = 0$
\item $w^T\frac{\partial^2 f_{\mu_0}}{\partial x^2}|_{x_0}(v,v) \neq 0$ and $w^T\frac{\partial^2 f_{\mu}}{\partial x \partial \mu}|_{x_0,\mu_0}v \neq 0$
\end{enumerate}
Then the phase portrait is topologically equivalent to the phase portrait of the logistic equation, i.e. we have a transcritical bifurcation about $x_0$ for $\mu=\mu_0$. Thus around $\mu=\mu_0$, there are two arcs of equilibria whose stability properties are exchanged when passing through $\mu_0$. Moreover, the set of equations $\dot x = f_\mu(x)$ which satisfy conditions $(1), (2)$ and $(3)$ above is generic in the space of smooth one-parameter families of vector fields with an equilibrium at $x_0$, $\mu_0$  with a zero eigenvalue.
\end{Theorem}

\section{Formation Control}\label{sec:formcontrol}

Let $G=(V,E)$ be a \emph{graph} with $n$ vertices --- that is $V = \lbrace v_1,v_2,\ldots,v_n \rbrace$  is an ordered set of vertices and $E \subset V \times V$ is a set of edges. The graph is said to be \emph{directed} if $(v_i,v_j) \in E$ does not imply that  $(v_j,v_i) \in E$.  We let $|E| = m $ be the cardinality of $E$. We call the \emph{outvalence} of a vertex the number of edges originating from this vertex.

Directed graphs are used to encode the \emph{information flow} in decentralized control problems. We follow the convention that an arrow leaving vertex $v_i$ for vertex $v_j$ means that agent $i$ measures the relative position---relative to its own location--- of agent $j$.

Assume that the edges are ordered. The mixed-adjacency matrix of a graph $G=(V,E)$ is a $|E| \times |V|$ matrix whose entry $(i,j)$ is $-1$ if edge $e_i$ originates from vertex $v_j$, $1$ if edge $e_i$ ends at vertex $v_j$ and $0$ otherwise:

\begin{Definition}[Mixed adjacency matrix] 
Given a directed graph $G=(V,E)$, its mixed adjacency matrix $A_m \in \R^{n \times m} $ is defined by 
$$A_{m,ij} = \left\lbrace \begin{matrix} -1 &\mbox{ if } e_i=(v_j,v_s),  v_s \in V \\ +1 &\mbox{ if } e_i=(v_k,v_j), v_k \in V\\ 0  & \mbox{ otherwise.}\end{matrix}\right.$$
\end{Definition}

The edge-adjacency matrix is a $|E| \times |E|$ matrix whose entry $(i,j)$ is $-1$ if edge $e_i$ and edge $e_j$ originate from the same vertex , $1$ if edge $e_i$ ends at the vertex where edge $e_j$ starts and $0$ otherwise. Notice that $A_{e,ij}$ is  zero if edge $e_i$ starts where edge $e_j$ ends and that the diagonal entries are  $-1$:

\begin{Definition}[Edge-adjacency matrix] 
Given a directed graph $G=(V,E)$, its edge-adjacency matrix $A_e \in \R^{m \times n} $ is defined by 
$$A_{e,ij} = \left\lbrace \begin{matrix} -1 &\mbox{ if } e_i=(v_s,v_t), e_j=(v_s,v_{t'},   v_s,v_t,v_{t'} \in V \\ 1 &\mbox{ if } e_i=(v_s,v_t), e_j=(v_t,v_{s'}), v_s,v_{s'},v_t \in V\\ 0  & \mbox{ otherwise.}\end{matrix}\right.$$
\label{def:ae}\end{Definition}

We will often encounter the matrix $A_m \otimes I$  where $\otimes$ is the Kronecker product and $I$ the two-by-two identity matrix.  In order to keep the notation simple, we  write $A_m^{(2)}$  for this Kronecker product.

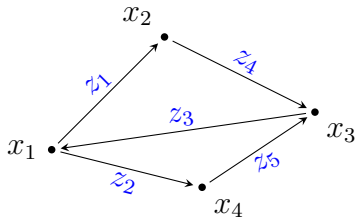
\begin{figure}
\begin{center}
\begin{tikzpicture} 

\node [fill=black,circle, inner sep=1pt,label=180:$x_1$] (1) at ( 0, 0) {};
\node [fill=black,circle, inner sep=1pt,label=135:$x_2$] (2) at (1.5 ,1.5) {};
  \node [fill=black,circle, inner sep=1pt,label=-45:$x_3$] (3) at (3.5 , .5) {};
\node [fill=black,circle, inner sep=1pt,label=-45:$x_4$] (4) at ( 2,-.5) {};

\draw [-stealth,  ] (1) -- (2) node [ midway,above,  sloped, blue] {$z_1$}; ;
\draw [-stealth, ] (1) -- (4) node [ midway,below,  sloped, blue] {$z_2$};;
\draw [-stealth, ] (3) -- (1)  node [ midway,above,  sloped, blue] {$z_3$};;
\draw [-stealth, ] (2) -- (3)  node [ midway,above,  sloped, blue] {$z_4$};;
\draw [-stealth, ] (4) -- (3) node [ midway,below,  sloped, blue] {$z_5$};;
\end{tikzpicture}\caption{The 2-cycles formation.}\label{fig:2cycle1}
\end{center}
\vspace{-1cm}
\end{figure}
\begin{Example}
The mixed-adjacency and edge-adjacency matrices of the 2-cycles of Figure~\ref{fig:2cycle1} are

\begin{equation}\label{eq:defAme2c} 
A_m = \left[\begin{array}{rrrr}
-1 & 1 & 0 & 0\\
0& -1& 1& 0\\
1&0 & -1 & 0\\
0& 0 &1 & -1\\
-1 & 0 &0 &1
\end{array}\right]\mbox{ and } A_e = \left[\begin{array}{rrrrr}
-1 & 1 & 0 & 0 & -1\\
0& -1& 1& 0 & 0\\
1&0 & -1 & 0& 1\\
0& 0 &1 & -1& 0\\
-1 & 0 &0 &1&-1
\end{array}\right].\end{equation}

respectively, where  edge $i$ is labelled by $z_i$ as in Figure~\ref{fig:2cycle1}. \hfill $\square$
\end{Example}

\subsection{Rigidity }
We briefly cover the fundamentals of rigidity and establish the relevant notation. We refer the reader to~\cite{graver93} for a more detailed presentation. 
We call a \emph{framework} an embedding of a graph in $\R^2$ endowed with the usual Euclidean distance, i.e. given $G=(V, E)$, a framework $p$ \emph{attached to a graph} $G$ is a mapping \begin{equation*} p: V \rightarrow \R^2.
\end{equation*}

We  write $x_i$ for $p(v_i)$. We define the  \emph{distance function} $\delta$ of a framework with $n$ vertices as \begin{multline*}
\delta(p): \R^{2n} \rightarrow \R_+^{n(n-1)/2}:(x_1, \ldots, x_n)  \rightarrow \frac{1}{2}\left[ \|x_1-x_2\|^2, \ldots,  \| x_2-x_3\|^2, \ldots,  \|x_{n-1}-x_n\|^2  \right],
\end{multline*} where $\R^+ = [0, \infty) $; i.e. $\delta(p)$ evaluates all the pairwise distances between edges. 
We denote by $\delta(p)|_E$ the restriction of the range of $\delta$ to edges in $E$.

For a graph $G$ with $m$ edges, we define the set of feasible edge-lengths \begin{equation*}\mathcal{L}=\left\lbrace d=(d_1,\ldots,d_m ) \in \R^m_+ \mbox{ for which }  \exists p \mbox{ with } \right.  \left. \delta(p(V))|_E=\sqrt{d} \right\rbrace,\end{equation*} where the square root of $d$ is taken entry-wise. Properties of this set and its relations to the number of ancillary equilibria are discussed in~\cite{belabbasSICOpart1}, but we observe that in the case of the 2-cycles, $\mathcal L$ is of dimension 5. We have taken the square root of $d$ for computational convenience. We denote by $\mathcal{L}_0$ the interior of $\mathcal L$.

The \emph{rigidity matrix}  $R$ of the framework is the Jacobian   $\frac{\partial \delta}{\partial x}$  restricted to the edges in $E$. We denote it by $R=\frac{\partial \delta}{\partial x}|_E$. 
The relevant definitions from rigidity theory are:
\begin{enumerate}
\item \emph{Static rigidity}: A framework is said to by statically rigid, or simply rigid, if for any $d \in \mathcal L$, there are only a finite number of frameworks,   modulo rotation and translation of the plane, such that $\delta|_E(p)=d$.
\item \emph{Infinitesimal rigidity}: A framework is said to be infinitesimally rigid if there are no vanishingly small motions of the vertices, modulo  rotations and translations of the plane, that keep the edge-length constraints satisfied. This translates into~\cite{graver93}:
$$\operatorname{rank} (\frac{\partial \delta}{\partial x}|_{E})=2n-3.$$ 

\item\emph{Minimal rigidity}: A framework is said to be {minimally rigid} if none of the  $m$ frameworks with $m-1$ edges obtained by removing one edge is rigid. 

\end{enumerate}

\subsection{Directed formation control}

We formalize in this section the type of control system considered in this work. We are given a graph $G=(V,E)$ which is assumed to be \emph{minimally rigid} with $|V|=n$ and $|E|=m$. Let $d \in \mathcal{L}$ be a feasible edge-lengths vector. The objective of the formation control problem is to find a decentralized control law, where the information flow is given by $G$,  that will stabilize the system around a framework with  inter-agent distances given by $d$.

In more detail, each agent with position $x_i \in \R^2$ is represented by a vertex $v_i$ in $V$. The dynamics of $x_i$ is allowed to depend only on the relative position of agents $x_k$ for which there is an edge originating from $v_i$ and ending at $v_k$: $$\dot x_i = u_i(x_k, x_l, \ldots), \mbox{ where } (v_i,v_k), (v_i,v_l), \ldots \in E.$$

In the case of directed control, it is easy to see that one cannot in general ask an agent to satisfy more than two edge lengths constraints~\cite{baillieulcdc03}. For a subtler analysis of the situation, we refer to~\cite{hendrickx07}. From now on, we always assume that $G$ has a maximum outvalence of two.

Given orderings of the edges and vertices of $G$, we define $$z_i=x_k-x_l$$ if edge $e_i$ links nodes $v_k$ to $v_l$. We define, with a slight abuse of  notation, $$e_i = z_i^Tz_i - d_i$$ the error in edge length. 

We can thus write the set $\mathcal{E}_d$ for formation control problems as $$\mathcal{E}_d =\lbrace x \in \R^{2n} | e_i(x)=0, \mbox{for all edges in } E \rbrace.$$
An important feature of the formation control problem is that it is defined up to a rigid transformation of the plane: if $x \in \R^{2n}$ is a framework of $G$,  frameworks obtained by a rotation and translation of $x$---we write them as $A\cdot x$, for $A \in SE(2)$, the special Euclidean group~\cite{belabbasSICOpart1}---are equivalent to $x$ from a formation control point of view. 
As a consequence, we can assume without loss of generality that the agents measure only the \emph{relative} positions of other agents. 

In most work on formation control, it is assumed that the agents' control law depends on the desired edge-lengths solely through the error in edge lengths $e_i$:
$$u = u(e_i)= u(z_i^Tz_i - d_i).$$

We consider here a broader class of systems by allowing the $u_i$'s to depend explicitly on the objective distances $d$. We distinguish this dependence of the control on $d$, which is a parameter by opposition to a dynamical variable, by using a semi-colon: $u=u(d_i;z_i^Tz_i-d_i)$.

We have: 
$$\dot x_i = u_i(d_i; e_i) z_i$$ in case agent $i$ follows a single agent and $$ \dot x_i = u_{1i}(d_k,d_l;e_k,e_l, z_k^Tz_l)z_k+ u_{2i}(d_k,d_l;e_k,e_l, z_k^Tz_l)z_l$$ in case agent $i$ follows two agents. For a complete justification of this model, see~\cite{belabbasSICOpart2}.

We established in~\cite{belabbasSICOpart1} a few conditions a feedback control law had to satisfy in order to yield a well-defined formation control system. We use them here to define a class of  feedback control law for formation control:

\begin{Definition} A feedback control law $u_i$ is \emph{compatible} with a formation control problem if
\begin{enumerate}
\item $u_i(d_j; e_j)$ is such that $u_i(d_j; 0)=0$ if agent $i$ has one co-leader.
\item $u_i(d_j,d_k; e_j,e_k,z_j \cdot z_k)$ is such that $u_i(d_j,d_k;0,0,z)=0$ for all $z$ if agent $i$ has two co-leaders.
\end{enumerate}
We accordingly define the class of controls $\mathcal{U}$ to be all twice differentiable control laws such that $u_i(d_i;e_i) = 0$ and $u_j(d_i,d_j;e_i,e_j,\cdot)=0$ have an isolated zero for $e_i=e_j=0$, depending on whether the agent has one or two co-leaders. ~\label{def:defcont}
\end{Definition}
From the above discussion, we conclude that:
\begin{Proposition} Let $\mathcal{E}_d$ be the set of design equilibria for a formation control problem with underlying information flow graph $G.$ If the  graph $G$ is \emph{rigid}, the set $\mathcal{E}_d/SE(2)$ is finite. In other words, the set of design equilibria is finite up to rigid transformations.
\end{Proposition}

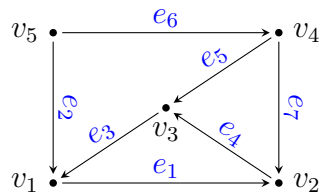
\begin{figure}[ht]
\begin{center}
\begin{tikzpicture} 

\node [fill=black,circle, inner sep=1pt,label=180:$v_1$] (1) at ( 0, 0) {};
\node [fill=black,circle, inner sep=1pt,label=0:$v_2$] (2) at (3 ,0) {};
 \node [fill=black,circle, inner sep=1pt,label=-90:$v_3$] (3) at (1.5 , 1) {};
\node [fill=black,circle, inner sep=1pt,label=0:$v_4$] (4) at ( 3,2) {};
\node [fill=black,circle, inner sep=1pt,label=180:$v_5$] (5) at ( 0,2) {};

\draw [-stealth,  ] (1) -- (2) node [ midway,above,  sloped, blue] {$e_1$}; 
\draw [-stealth, ] (5) -- (1) node [ midway,above,  sloped, blue] {$e_2$};
\draw [-stealth, ] (3) -- (1)  node [ midway,above,  sloped, blue] {$e_3$};
\draw [-stealth, ] (2) -- (3)  node [ midway,above,  sloped, blue] {$e_4$};
\draw [-stealth, ] (4) -- (3) node [ midway,above,  sloped, blue] {$e_5$};
\draw [-stealth, ] (5) -- (4) node [ midway,above,  sloped, blue] {$e_6$};
\draw [-stealth, ] (4) -- (2) node [ midway,above,  sloped, blue] {$e_7$};

\end{tikzpicture}\caption{The subgraphs with vertices $\lbrace v_1, v_2, v_3 \rbrace$ and edges $\lbrace e_1, e_3, e_4 \rbrace$ or vertices $\lbrace v_1, v_2, v_3, v_4\rbrace$ and edges $\lbrace e_1, e_3, e_4, e_5, e_7 \rbrace$  define  subformations. The subgraph with vertices $V' = \lbrace v_1, v_4, v_5 \rbrace$ does not, for any set of edges, since outgoing edges point to vertices not in $V'$.}\label{fig:subformation}
\end{center}
\end{figure}
We end this section with the definition of \emph{subformation}, illustrated in Figure~\ref{fig:subformation}: given a graph $G=(V,E)$ underlying a formation, a graph $G'=(V',E')$ underlies a subformation if $G'$ is a subgraph of $G$ where all outgoing edges from vertices in $V'$ are included in $E'$.

\subsection{The 2-cycles formation}

\begin{figure}[]
\begin{center}
\subfloat[]{
\begin{tikzpicture}[scale=.7] 
\node [fill=black,circle, inner sep=1pt,label=135:$x_1$] (1) at ( 0, 0) {};
\node [fill=black,circle, inner sep=1pt,label=-135:$x_2$] (2) at (-1 ,0) {};
\node [fill=black,circle, inner sep=1pt,label=-45:$x_3$] (3) at (0 ,-1) {};
\node [fill=black,circle, inner sep=1pt,label=-45:$x_4$] (4) at (1 ,1) {};

\draw [-stealth ] (1) -- (2);
\draw [-stealth] (3) -- (1);
\draw [-stealth] (4) -- (3);
\draw [-stealth] (2) -- (3);
\draw [-stealth] (1) -- (4);

\end{tikzpicture}

} 
\subfloat[]{
\begin{tikzpicture}[scale=.7] 
\node [fill=black,circle, inner sep=1pt,label=90:$x_1$] (1) at ( 0, 0) {};
\node [fill=black,circle, inner sep=1pt,label=0:$x_2$] (2) at (1 ,0) {};
\node [fill=black,circle, inner sep=1pt,label=-45:$x_3$] (3) at (0 ,-1) {};
\node [fill=black,circle, inner sep=1pt,label=-45:$x_4$] (4) at (1 ,1) {};

\draw [-stealth ] (1) -- (2);
\draw [-stealth] (3) -- (1);
\draw [-stealth] (4) -- (3);
\draw [-stealth] (2) -- (3);
\draw [-stealth] (1) -- (4);

\end{tikzpicture}
} 
\subfloat[]{
\begin{tikzpicture}[scale=.7] 
\node [fill=black,circle, inner sep=1pt,label=45:$x_1$] (1) at ( 0, 0) {};
\node [fill=black,circle, inner sep=1pt,label=0:$x_2$] (2) at (1 ,0) {};
\node [fill=black,circle, inner sep=1pt,label=-135:$x_3$] (3) at (0 ,-1) {};
\node [fill=black,circle, inner sep=1pt,label=180-45:$x_4$] (4) at (-1 ,1) {};

\draw [-stealth ] (1) -- (2);
\draw [-stealth] (3) -- (1);
\draw [-stealth] (4) -- (3);
\draw [-stealth] (2) -- (3);
\draw [-stealth] (1) -- (4);

\end{tikzpicture}
} 
\subfloat[]{
\begin{tikzpicture}[scale=.7] 
\node [fill=black,circle, inner sep=1pt,label=90:$x_1$] (1) at ( 0, 0) {};
\node [fill=black,circle, inner sep=1pt,label=135:$x_2$] (2) at (-1 ,0) {};
\node [fill=black,circle, inner sep=1pt,label=-45:$x_3$] (3) at (0 ,-1) {};
\node [fill=black,circle, inner sep=1pt,label=135:$x_4$] (4) at (-1 ,1) {};

\draw [-stealth ] (1) -- (2);
\draw [-stealth] (3) -- (1);
\draw [-stealth] (4) -- (3);
\draw [-stealth] (2) -- (3);
\draw [-stealth] (1) -- (4);
\end{tikzpicture}
}\end{center}
\caption{\small Four formations in the plane that are not equivalent under rotations and translation and that have the same corresponding edge lengths. $(a)$ is the mirror-symmetric of $(c)$ and $(b)$ is the mirror-symmetric of $(d)$.}
\label{fig:4formations}
\vspace{-.6cm}
\end{figure}
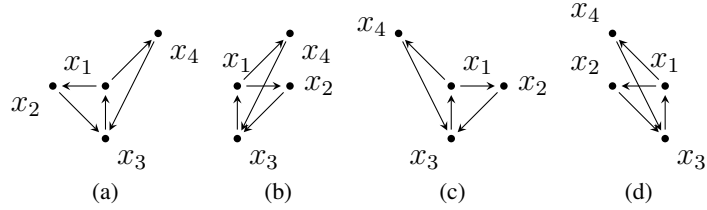
The 2-cycles is the formation represented in Figure~\ref{fig:s2c}. Let $x_i \in \R^2$, $i=1\ldots 4$ represent the position of $4$ agents in the plane. We define the vectors 
\begin{equation}
\label{eq:defz}
z_1 =x_2-x_1;\ 
z_2 = x_3-x_2;\  
z_3 = x_1-x_3; \ 
z_4 =  x_3-x_4;\ 
z_5 = x_4-x_1
\end{equation}

Hence a general control law for such a system is
\begin{equation}\label{eq:dyns}
\left\lbrace\begin{array}{rcl}
\dot x_1 &=& u_{11}(d_1,d_5;e_1,e_5,z_1^Tz_5) z_1+ u_{12}(d_1,d_5;e_1,e_5,z_1^Tz_5) z_5 \\
\dot x_2 &=& u_2(d_2;e_2) z_2 \\
\dot x_3 &=& u_3(d_3;e_3) z_3 \\
\dot x_4 &=& u_4(d_4;e_4) z_4
\end{array}\right.
\end{equation}

The set of design equilibria $\mathcal{E}_d$ for the 2-cycles is of cardinality $4$, up to rigid transformations, since there are four frameworks in the plane for which $e_i=0$; they are depicted in Figure~\ref{fig:4formations}.

In general,  the set $\mathcal{E}_a$ of ancillary equilibria depends  on the choice of feedbacks $u_i \in \mathcal U$. Due to the invariance and decentralized nature of the system, some configurations belong to $\mathcal{E}_a$ for all elements of $\mathcal U$:

\begin{Proposition} The set $\mathcal{E}$ contains, in addition to the equilibria in $\mathcal{E}_d$,  the  frameworks characterized by
\begin{enumerate}
\item $z_i=0$ for all $i$,  which corresponds to having all the agents superposed.
\item all $z_i$ are aligned, which corresponds to having all agents on the same one-dimensional subspace in $\R^2$. These frameworks form a three dimensional invariant subspace of the dynamics.
\item  $e_2=e_3=e_4=0$, $z_1$ and $z_5$ are aligned and so that $$u_1(e_1,e_5, z_1^T z_5) \|z_1\|=\pm u_5(e_1,e_5,z_1^T z_5) \|z_5\|,$$ where the sign depends on whether $z_1$ and $z_5$ point in the same or opposite directions. 
\end{enumerate}
\end{Proposition}\label{prop:allequ}

This result is straightforward from an inspection of Equation~\eqref{eq:dyns}. Frameworks of type 2 above are not infinitesimally rigid and  they define an invariant submanifold of the dynamics. Frameworks of type 3 appear whenever a vertex of the information flow graph has an outvalence of two.

\subsection{Singular formations for $n=4$ agents.}

Let us gather the $m$ vectors $z_i \in \R^2$ in $z=(z_1,z_2, \ldots, z_m) \in \R^{2m}$. We set $$D(z) = \left[\begin{matrix}
z_1^T & 0 & 0& \ldots & 0\\
0 &z_2^T& 0 & \ldots & 0 \\
0 & \ldots & & \ddots &\vdots \\
0 & 0 &\ldots & 0 &z_m^T
\end{matrix}\right].$$ Hence $D(z) \in \R^{m\times 2m}$.

In order to use Sotomayor's theorem, we single out a particular type of frameworks which, even though they are infinitesimally rigid, show a certain degree of degeneracy.  We first observe that,  in general, the angle between $z_1$ and $z_5$ is not uniquely determined  by the edge lengths. We define $\mathcal{S}$ to be set of edge lengths such that \emph{at least} one of the four frameworks corresponding to $d$ has $z_1$ parallel to $z_5$ with the notation of Figure~\ref{fig:illS}: $$\mathcal{S}= \lbrace d \in \mathcal{L} \mbox{ s.t. } z_1 \mbox{ parallel } z_5 \mbox{ for one framework at least.} \rbrace$$ and $\mathcal{S}_0 = \mathcal{S} \cap \mathcal{L}_0$.

\begin{figure}[]
\begin{center}
\subfloat{
\begin{tikzpicture}[scale=.7] 

\node [fill=black,circle, inner sep=1pt,label=180:$x_1$] (1) at ( 0, 0) {};
\node [fill=black,circle, inner sep=1pt,label=135:$x_2$] (2) at (-1.5 ,-1.5) {};
  \node [fill=black,circle, inner sep=1pt,label=-45:$x_3$] (3) at (0 , -2.5) {};
\node [fill=black,circle, inner sep=1pt,label=-45:$x_4$] (4) at ( 1,-1.5) {};

\draw [-stealth, , ] (1) -- (2) node [ midway,above,  sloped, blue] {$z_1$};
\draw [-stealth, ] (1) -- (4) node [ midway,above,  sloped, blue] {$z_5$};
\draw [-stealth, ] (2) -- (3) node [ midway,above,  sloped, blue] {$z_2$};
\draw [-stealth, ] (4) -- (3) node [ midway,above,  sloped, blue] {$z_4$};
\draw [-stealth, ] (3) -- (1) node [ midway,above,  sloped, blue] {$z_3$};
\end{tikzpicture}}\qquad
\subfloat{
\begin{tikzpicture}[scale=.8]
\node [fill=black,circle, inner sep=1pt,label=180:$x_1$] (1) at ( 0, 0) {};
\node [fill=black,circle, inner sep=1pt,label=135:$x_2$] (2) at (-1.5 ,.75) {};
  \node [fill=black,circle, inner sep=1pt,label=-45:$x_3$] (3) at (0 , -2.5) {};
\node [fill=black,circle, inner sep=1pt,label=-45:$x_4$] (4) at ( .75,-.5*3/4) {};

\draw [-stealth, , ] (1) -- (2) node [ midway,above,  sloped, blue] {$z_1$};
\draw [-stealth, ] (1) -- (4) node [ midway,above,  sloped, blue] {$z_5$};
\draw [-stealth, ] (2) -- (3) node [ midway,above,  sloped, blue] {$z_2$};
\draw [-stealth, ] (4) -- (3) node [ midway,above,  sloped, blue] {$z_4$};
\draw [-stealth, ] (3) -- (1) node [ midway,above,  sloped, blue] {$z_3$};
\end{tikzpicture}
}
\end{center}
\caption{\small The formation in (a) is such that $(\|z_1\|,\ldots,\|z_5\|) \notin \mathcal{S}$, whereas $(\|z_1\|,\ldots,\|z_5\|) \in \mathcal{S}$ for the formation depicted in (b) }
\label{fig:illS}
\vspace{-.5cm}
\end{figure}
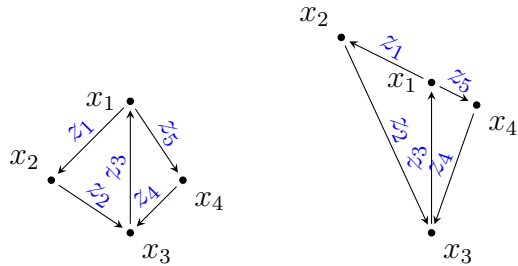

 We will need the following properties of $\mathcal S$:

\begin{Lemma}\label{lem:propS}
The following properties of $\mathcal{S}$ hold:
\begin{enumerate} 
\item $\mathcal{S}$ is of codimension one is $\mathcal{L}$
\item The formations corresponding to edge lengths in $\mathcal{S}_0$ are infinitesimally rigid.
\end{enumerate}
\end{Lemma}

\begin{proof}
For the first part, observe that we can parametrize $\mathcal{S}$ by first choosing a feasible $d_1, d_2, d_3$ yielding a triangle $x_1, x_2, x_3$ and one additional parameter giving the signed length of $z_5 $, with the sign referring to $z_5$ going in the same direction as  $z_1$ or the opposite direction. Because the set of feasible $d_1, d_2, d_3$ is included in $\mathcal{L}$, we see that we need $4$ parameters to describe a formation in $\mathcal{S}$ and hence it is of codimension one.

For the second part, we have that the rigidity matrix of the 2-cycles is given by $R$:
{\small}
\begin{equation}
\label{eq:rigtc}
R= \left[ \begin{matrix}
z_1^T & -z_1^T & 0 & 0 \\
0 & z_2^T &-z_2^T &0 \\
-z_3^T & 0 & z_3^T & 0 \\
0 & 0 & -z_4^T & z_4^T\\
z_5^T & 0 & 0 & -z_5^T
\end{matrix} \right].
\end{equation}

Some simple algebra shows that one has \begin{equation}
\label{eq:RZ}
R = D(z)A^{(2)}_m
\end{equation}
where we recall that $A_m$ is the mixed adjacency matrix. 
In the case of the 2-cycles, the mixed adjacency matrix $A_m \in \R^{5 \times 4}$ is of rank $3$. The cokernel\footnote{The cokernel of a linear map $f:A \rightarrow B$ is the quotient space $B/\operatorname{im}(f)$. Its coimage is $A / \operatorname{ker}(f)$.} of $A_m$ is spanned by $[0, 0, 1, 1 ,1 ]^T$ and $[1,1,1,0,0]^T$. Hence, the cokernel of $A_m^{(2)}$ is four dimensional and spanned by the vectors $[0, 0, 1, 1 ,1 ]^T \otimes [1, 0]^T$, $[0, 0, 1, 1 ,1 ]^T \otimes [0, 1]^T$, $[1, 1, 1, 0 ,0 ]^T \otimes [1, 0]^T$ and $[1, 1, 1, 0 ,0 ]^T \otimes [0, 1]^T$

The matrix $D(z)$ is of full rank unless  $z_i=0$ for some $i$, which corresponds to two agents superposed.  We thus have that $D(z)$ is of full rank for formations in $\mathcal{S}_0$. Because $D(z)$ is of full row rank,  $R$ is of full (row) rank if $A_m^{(2)}$ projects \emph{onto} the coimage of $D(z)$, which is easy to see by inspection. 
\end{proof}

\section{The 2-cycles is not almost surely stabilizable}\label{sec:form}
\subsection{Statement of the main result}
The main result of this paper is to show that given any robust control law, there is an open set of target configurations that are not almost surely stabilizable for the 2-cycles; equivalently, to show that for any control law that robustly stabilizes a design equilibrium, there is an open set of $d$'s in $\mathcal{L}$ for which  an \emph{ancillary} equilibrium is stable.

\begin{Theorem}~\label{th:notstab} The 2-cycles formation is not robustly almost-sure stabilizable  for an open set of design frameworks.
\end{Theorem}

The result readily extends to formations that contain the 2-cycles as a subformation:

\begin{Corollary} Any formation that contains the 2-cycles as a subformation is not robustly almost-sure  stabilizable for an open set of design frameworks.
\end{Corollary}

\subsection{Proof of the main result}

We rewrite the dynamics in terms of the $z$ variables

\begin{equation}
\label{eq:dynr}
 \left\lbrace \begin{array}{l}
\dot z_1 = u_2 z_2 - u_1 z_1- u_5z_5\\
\dot z_2 = u_3 z_3 - u_2 z_2\\
\dot z_3 =u_1 z_1+ u_5 z_5-u_3 z_3 \\
\dot z_4 =u_3 z_3  - u_4 z_4 \\
\dot z_5 =u_4 z_4 - u_1 z_1- u_5z_5
\end{array} \right.
\end{equation}
 where the dependence of the $u_i$ on $d_i$ and $z_i$'s is given in Equation~\eqref{eq:dyns}. We denote by $F(z)$ the right-hand side of Equation~\eqref{eq:dynr} and set $$\mathcal{F} = \left\lbrace F(z) \ | \ u_i \in \mathcal{U} \right\rbrace. $$

We now show that  the system of Equation~\ref{eq:dynr} has the logistic equation as normal form at $\mathcal{S}_0$.

\begin{Theorem}\label{th:tech1} A transcritical bifurcation at frameworks with $d$ in an open set in $\mathcal{S}_0$ is generic for systems in $\mathcal{F}$. 
\end{Theorem}
We will prove Theorem~\ref{th:notstab} as a corollary of this result. We prove Theorem~\ref{th:tech1} in several steps. 

\begin{Proposition}\label{prop:jaczerobif}
Let $d\in \mathcal{S}_0$.  There is a non-zero vector $w \in \R^{10}$ such that $w^T\frac{\partial F}{\partial z}|_{e_i=0,d}= w^T \frac{\partial F}{\partial d}|_{d}=0$ for at least one framework attached to $d$.
\end{Proposition}

We denote by $u_{x}(d_1,d_2;x,y,z)$ the  derivative of $u$ with respect to $x$, and similar definitions hold for $u_y$ and $u_z$. We need the following lemmas:

\begin{Lemma}\label{prop:jaclow}
Set $$z'_i= 2 (u_{1x}z_i+u_{2x}z_j)$$ and $$z'_j=2(u_{1y}z_i+u_{2y}z_j)$$ if $z_i$  originates from vertex with outvalence two and outgoing edges $z_i$ and $z_j$, and $$z'_i =2 u_x z_i$$ if $z_i$ originates from a vertex with outvalence one. 
The Jacobian at a design equilibrium of a formation control system is 
\begin{equation}\label{eq:defJ}
\frac{\partial F}{\partial z }=A_e^{(2)}D(z')^TD(z).
\end{equation} 
\end{Lemma}

\begin{proof}
We first observe that 
\begin{eqnarray*}
\frac{\partial }{\partial z_i}u_1(e_i,e_j,z_i^Tz_j)z_i & = & 2u_{1x} z_iz_i^T+u I_2+2u_{1z}z_iz_j^T \\
\frac{\partial}{\partial z_j} u_1(e_i,e_j,z_i^Tz_j)z_i & = & 2u_{1y} z_iz_j^T+2u_{1z} z_iz_j^T\\ 
\frac{\partial}{\partial z_i} u_2(e_i,e_j,z_i^Tz_j)z_j & = & 2u_{2x} z_jz_i^T+2u_{2z} z_jz_i^T\\
\frac{\partial}{\partial z_j} u_2(e_i,e_j,z_i^Tz_j)z_j & = & 2u_{2y} z_jz_j^T+2u_{2z} z_jz_j^T
\end{eqnarray*}
where we omitted the arguments of the functions on the right-hand side.
Recall from Definition~\ref{def:defcont} that  $u_z=u=0$ at a design equilibrium. Hence, if $z_i$ originates from a vertex with two outgoing edges with $$\dot z_i=F_i(z) = \ldots - u_1(e_i,e_j,z_i^Tz_j)z_i- u_2(e_i,e_j,z_i^Tz_j)z_j$$ then:
\begin{eqnarray*}
\frac{\partial F_i}{\partial z_i}&=& -2u_{1x} z_iz_i^T-2u_{2x} z_jz_i^T= - z_i'z_i^T\\
\end{eqnarray*}
Similarly, $$\frac{\partial F_i}{\partial z_j} = -2z'_jz_j^T.$$
If $z_k$ originates from an agent with a single leader, we have:
$$\frac{\partial F_k}{\partial z_k}=u_x z_jz_j^T=z'_jz_j^T.$$ In general, if $z_k$ appears in $F_l$, then $$\frac{\partial F_l}{\partial z_k} = \pm z'_kz_k^T$$ where the sign is negative if both $z_l$ and $z_k$ are leaving the same vertex and positive if $z_l$ is leaving the vertex to which $z_k$ leaves. Putting the equations above together and recalling Definition~\ref{def:ae} of $A_e$, we get the result.
\end{proof}

We define \begin{equation}\left\lbrace \begin{array}{rcl}
z''_1 &=&z'_1 +  \frac{\partial u_1}{\partial d_1} z_1 + \frac{\partial u_5}{\partial d_1}z_5\\
z''_2 &=&z'_2+  \frac{\partial u_2}{\partial d_2} z_2\\
z''_3 &=&z'_3+  \frac{\partial u_3}{\partial d_3} z_3\\
z''_4 &=&z'_4+  \frac{\partial u_4}{\partial d_4} z_4\\
z''_5 &=&z'_5 +  \frac{\partial u_1}{\partial d_5} z_1 + \frac{\partial u_5}{\partial d_5}z_5
\end{array}\right.
\end{equation}


\begin{Lemma} The Jacobian of $F$ with respect to the parameters $d$ at a design equilibrium is given by
\begin{equation}\label{eq:dfddapp}
\frac{\partial F}{\partial d} = A_e^{(2)} D(z'')^T.
\end{equation}

\end{Lemma}

\begin{proof}
We have $$\frac{\partial F_1}{\partial d_1} = -(\frac{\partial u_1}{\partial d_1}+u_{1x})z_1 - (\frac{\partial u_5}{\partial d_1}+u_{5x})z_5$$ and similar relations for the other entries $\frac{\partial F_i}{\partial d_j}$.  Some algebraic manipulations yield  the result.
\end{proof}

\begin{Lemma}~\label{lem:dfxdfd} Let $d \in \mathcal{S}_0$, then  $w$ is a left eigenvector of  $\frac{\partial F}{\partial z}|_{e_i=0,d}$  with eigenvalue $0$ if and only if $w^T \frac{\partial F}{\partial d}|_{e_i=0,d} = 0$.
\end{Lemma}
\begin{proof}
We claim that when $z_1$ is parallel to $z_5$, we can find invertible diagonal matrices $D_1$ and $D_2$ such that
\begin{equation}\label{eq:zppdz}
D(z')=D_1D(z) = D_2D(z'').
\end{equation}Indeed, for $i=2,3,4$, it is immediate from the definitions of $z_i$, $z_i'$ and $z_i''$ that there exists $\alpha_, \beta_i \neq 0$  such that $$z_i= \alpha_i z'_i = \beta_i z''_i, i=2,3,4. $$ 
The $\alpha_i, \beta_i$ are the entries of $D_1$ and $D_2$. From Lemma~\ref{prop:jaclow}, we have
\begin{equation}
\label{eq:reppropjaclow}
\frac{\partial F}{\partial z} = A_e^{(2)} D(z')^TD(z)
\end{equation}
Putting Equations~\eqref{eq:dfddapp},~\eqref{eq:zppdz},~\eqref{eq:reppropjaclow} together, we obtain that $$\frac{\partial F}{\partial d} D_2 D(z)= \frac{\partial F}{\partial z}.$$

If all the $z_i$ are non-zero, then $D_2D(z)$ is of full rank and we conclude that \begin{equation}
\label{eq:kerndfddapp}
w^T \frac{\partial F}{\partial d} = 0 \Leftrightarrow w^T \frac{\partial F}{\partial z} = 0.\end{equation}
\end{proof}

Formation control systems are invariant under an action of the Euclidean group $SE(2)$ on $\R^2$. Hence the Jacobian of the dynamics in the $x$ variables will always have at an equilibrium  three zero eigenvalues corresponding to the three dimensions of $SE(2)$. For the dynamics in the $z$ variables, there are additional zeros from the redundancy of the $z$. For example, described in the $z$ coordinates, the 2-cycles has 10 dimensions, but four degrees of freedom are lost to $z_1+z_2+z_3=z_3+z_4+z_5=0$ and one additional degree of  freedom is lost to the $SE(2)$ invariance, since the invariance under translation is taken into account in the $z$ variables and only the invariance under rotations, which is one dimensional, remains. The following result addresses this point. 

\begin{Corollary}\label{cor:eigdefJ}Let $G$ be the graph of a minimally  rigid formation with edge adjacency matrix $A_e$.  The eigenvalues of the Jacobian of $F(z)$ at a non-singular design equilibrium are the eigenvalue zero with algebraic multiplicity $2n-3$  and the eigenvalues of \begin{equation}\label{eq:defJ}J = D(z)A_eD(z')^T.\end{equation}
\end{Corollary}

\begin{proof}The result is a consequence of Theorem 1.3.20 in~\cite{hornjohnson90}  applied to  Lemma~\ref{prop:jaclow}.
\end{proof}

In the remainder of the paper, whenever we refer to eigenvalues and eigenvectors of $F$, we will refer to the eigenvalues and eigenvectors that do not correspond to the redundant description of the system, and to which we have access thanks to Corollary~\ref{cor:eigdefJ}.

\begin{Corollary}[Singular formations]\label{cor:singform}
 Let $d \in \mathcal S_0$. The Jacobian of the 2-cycles formation is generically of corank $1$ for at least one framework attached to  $d$. 
 \end{Corollary}
\vspace{.2cm}
\begin{proof}
A direct computation using Corollary~\ref{cor:eigdefJ}  and the edge-adjacency matrix of the 2-cycles gives 
\begin{equation*}J=D(z)A_eD(z')^T = \left[ \begin{array}{rrrrr}
-z_1^T z'_1 & z_1^T z'_2 & 0 & 0 & -z_1^T z'_5 \\
0 & -z_2^T z'_2 & z_2^T z'_3 & 0 & 0\\
z_3^T z'_1 & 0 & -z_3^T z'_3 & 0 & z_3^T z'_5 \\
0 & 0 & z_3^T z'_4 & -z_4^T z'_4 & 0 \\
-z_1^T z'_5 & 0 & 0 & z_4^T z'_5 & -z_5^T z'_5
\end{array}\right].\end{equation*}

By Corollary~\ref{cor:corgent} in the appendix, $u_i' \neq 0$ generically when it vanishes, which correspond to design equilibria. Hence $z'_i$ are generically non-zero. For the framework attached to $d \in \mathcal{S}_0$ such that $z_1$ is parallel to $z_5 $,  the first and last column are multiples of each other, and it is easy to see that the first four columns are linearly independent. The corank is higher if, in addition,  one of the $z_i$ is zero.
\end{proof}

We now prove Proposition~\ref{prop:jaczerobif}.

\begin{proof}[Proof of Proposition~\ref{prop:jaczerobif}]
Consider a framework with $d \in \mathcal{S}_0$ and $z_1$ parallel to $z_5$. From Corollary~\ref{cor:singform}, we know that $\frac{\partial F}{\partial  z}$ is generically of rank $4$. Let $w$ be an eigenvector corresponding to the zero eigenvalue. We conclude using Lemma~\ref{lem:dfxdfd} that $w^T \frac{\partial F}{\partial d} = 0$.
  \end{proof}

For clarity of the exposition, we now restrict ourselves to the system
\begin{equation}
\left\lbrace \begin{matrix}
\dot x_1 &=& u(d_1;e_1) z_1 +u(d_5;e_5) z_5 \\
\dot x_2 &=& u(d_2;e_2) z_2  \\
\dot x_3 &=& u(d_3;e_3) z_3  \\
\dot x_4 &=& u(d_4;e_4) z_4 
\end{matrix}\right. 
\end{equation} The preliminary results, proved in greater generality, make the extension of the proof below to the more general system easy.

\begin{proof}[Proof of Theorem~\ref{th:tech1}]
Fix $d_0=(d_1,d_2,d_3,d_4,d_5) \in \mathcal{S}_0$. We consider the \emph{one parameter system} where only $\mu \in \R$ is allowed to vary:
\begin{equation}
\left\lbrace \begin{matrix}
\dot x_1 &=& u(d_1;e_1) z_1 +u(d_5;e_5) z_5 \\
\dot x_2 &=& u(d_2;e_2) z_2 \\
\dot x_3 &=& u(d_3+\mu; z_3^Tz_3 - (d_3+\mu)) z_3  \\
\dot x_4 &=& u(d_4;e_4) z_4 
\end{matrix}\right. 
\end{equation}  and the corresponding equations in $z$ variables. We  prove that conditions $(1), (2)$ and $(3)$ of Theorem~\ref{th:soto} are generic for $\mathcal{F}$.
From Corollary~\ref{cor:singform} and the fact that $u' \neq 0$  generically at a zero of $u$ (Corollary~\ref{cor:corgent}),  we know that the Jacobian of the 2-cycles at $\mathcal{S}_0$ has a unique zero eigenvalue (that is not a result of the redundant description of the system) generically for $F \in \mathcal{F}$. Hence condition $(1)$ is verified. Condition $(2)$ follows from Proposition~\ref{prop:jaczerobif}.

Condition $(3)$ takes the following form: for $w$ and $v$ left and right eigenvectors of $\frac{\partial F}{\partial z}$ respectively, $$w^T \frac{\partial^2 F}{\partial z^2}(v,v) = \sum_{ijk} \frac{\partial^2 F_i}{\partial z_j\partial z_k} w_iv_jv_k \neq 0,$$ where $w$ and $v$ depend on the design equilibrium. Using the relations established in Lemma~\ref{prop:jaclow} for partial derivatives of $F_i$,  the triple sum in the above Equation can be explicitly evaluated. Rearranging terms and introducing the constants $c_{1}$, $c_{2}$ (which depend on the design equilibrium), we obtain a linear combination of first and second derivatives of  $u$:
$$w^T \frac{\partial^2 F}{\partial z^2}(v,v) = c_{1} u_{x} +  c_{2} u_{xx}.
$$ In order to verify that there is an open set of design equilibria for which the functions $c_{1}$ and $c_{2}$ are non-zero, since the functions are continuous, it suffices to find one design equilibrium at which it is the case. It is easily verified, for example, at the design equilibrium with edge lengths corresponding to the framework with $x_1=(0,0), x_2=(-2,1), x_3=(0,-1)$ and $x_4=(1/2,-1/4)$.
 
On that open set, we thus have $w^T\frac{\partial^2 F}{\partial z^2 }(v,v)=0$ only if $$c_{1} u_{x} +  c_{2} u_{xx}=0$$ when $u$ vanishes (since, by definition, $u$ vanishes at design equilibria). Let $C \subset J^2(\R,\R)$ be defined by the equations $$c_{1} u_{x} +  c_{2} u_{xx}=0$$ and $$u=0.$$  $C$ is of codimension 2 in $J^2(\R,\R)$ whereas the image of the 2-jet extension of $u$ is of dimension 1. Hence it is transversal   to $C$ if and only if  $u\neq 0$ or $c_{1} u_{x} +  c_{2} u_{xx}\neq 0$. We conclude using Thom's transversality Theorem~\ref{th:thom} that $w^T\frac{\partial^2 F}{\partial z^2}(v,v)$ is generically non-zero.

Using a similar reasoning as above, we can conclude that $w^T\frac{\partial^2 F}{\partial x \partial d }v$ is generically non-zero.  \end{proof}

\begin{proof}[Proof of Theorem~\ref{th:notstab}]

\begin{figure}[]
\begin{center}
\begin{tikzpicture}[scale=.9]
\draw[help lines, step=25pt,very thin, dotted](-2,-2) grid (2,2);

\draw[, blue,dashed] (-2,0) -- (0,0); \node (-1.4,0) [pin=-2:stable]{} ;\node at (-1.9,.3) {$\mathcal{E}_a$};
\draw[dashed, ,magenta,->] (0,0) -- ( 2,0); \node [black] at  (2.5,0) {$\mu$};

\draw[, black,->] (0,-2) -- (0,2) node (0,2.2) {Equilibria};

\draw[ , red] (-1.9,-1.5) -- (0,0); \node (1.9,0) [black,pin=275:unstable]{};\node at (1.9,1.7) {$\mathcal{E}_d$};
\node at (-1.2,-0.92) [black,pin=180:unstable]{};
\draw[, blue] (0,0) -- ( 1.9,1.5) ;
\node at (1.9,1.5) [black,pin=0:stable]{};
\node [fill=black,inner sep=1pt,label=-45:$0$] at (0,0) {}; 

\node (m1) [fill=black,inner sep=1pt,label=-45:$\mu_1$] at (-1.10,0) {}; 
\draw[dotted, black] (-1.1,0) -- (-1.1,-.87) {};
\node [fill=black,inner sep=1pt] at (-1.10,-.87) {};
\node (m2) [fill=black,inner sep=1pt,label=225:$\mu_2$] at (1.10,0) {}; 
\draw[dotted, black] (1.1,0) -- (1.1,.87) {};
\node [fill=black,inner sep=1pt] at (1.10,.87) {};

\draw [thin,-stealth ] (m1) -- +(-.5,.8);
\draw [thin,-stealth ] (m2) -- +(.5,-1);
\draw [thin,-stealth ] (-1.10,-.86) -- +(-.4,-.6);
\draw [thin,-stealth ] (1.10,.86) -- +(.4,.6);

\begin{scope}[yshift=4.4cm, xshift=-3cm]
\node [fill=black,circle, inner sep=1pt,label=90:$ x_1$] (1) at ( 0, 0) {};
\node [fill=black,circle, inner sep=1pt,label=135:$ x_2$] (2) at (-1.5 ,-.3) {};
\node [fill=black,circle, inner sep=1pt,label=180:$ x_3$] (3) at (0 ,-2) {};
\node [fill=black,circle, inner sep=1pt,label=45:$ x_4$] (4) at (1.5 ,.3) {};

\draw [-stealth,red ] (1) -- (2)node[ midway,sloped,above,font=\scriptsize]{$ d_1+\varepsilon_1$};;
\draw [-stealth] (3) -- (1)node[ midway,sloped,above,font=\scriptsize]{$d_3+\mu_1$}; ;
\draw [-stealth] (4) -- (3)node[ midway,sloped,below,font=\scriptsize]{$d_4$};;
\draw [-stealth] (2) -- (3)node[ midway,sloped,below,font=\scriptsize]{$d_2$};;
\draw [-stealth,red] (1) -- (4)node[ midway,sloped,above,font=\scriptsize]{$ d_5+\varepsilon_2$};;
\end{scope}
\begin{scope}[yshift=- 2.8cm, xshift=3cm]
\node [fill=black,circle, inner sep=1pt,label=90:$x_1$] (1) at ( 0, 0) {};
\node [fill=black,circle, inner sep=1pt,label=135:$x_2$] (2) at (-1 ,-.2) {};
\node [fill=black,circle, inner sep=1pt,label=180:$x_3$] (3) at (0 ,-2.2) {};
\node [fill=black,circle, inner sep=1pt,label=45:$x_4$] (4) at (1 ,.2) {};

\draw [-stealth,red ] (1) -- (2)node[ midway,sloped,above,font=\scriptsize]{$d_1+\varepsilon_3$};
\draw [-stealth] (3) -- (1)node[ midway,sloped,above,font=\scriptsize]{$d_3+\mu_2$}; ;
\draw [-stealth] (4) -- (3)node[ midway,sloped,below,font=\scriptsize]{$d_4$};
\draw [-stealth] (2) -- (3)node[ midway,sloped,below,font=\scriptsize]{$d_2$};
\draw [-stealth,red] (1) -- (4)node[ midway,sloped,below,font=\scriptsize]{$d_5+ \varepsilon_4$};
\end{scope}
\begin{scope}[yshift= 4.4cm, xshift=3cm]
\node [fill=black,circle, inner sep=1pt,label=90:$x_1$] (1) at ( 0, 0.2) {};
\node [fill=black,circle, inner sep=1pt,label=135:$x_2$] (2) at (-1.2 ,-.1) {};
\node [fill=black,circle, inner sep=1pt,label=180:$x_3$] (3) at (.0 ,-2.2) {};
\node [fill=black,circle, inner sep=1pt,label=45:$x_4$] (4) at (1.2 ,.1) {};

\draw [-stealth ] (1) -- (2)node[ midway,sloped,above,font=\scriptsize]{$d_1$};
\draw [-stealth] (3) -- (1)node[ midway,sloped,above,font=\scriptsize]{$d_3+\mu_2$}; ;
\draw [-stealth] (4) -- (3)node[ midway,sloped,below,font=\scriptsize]{$d_4$};
\draw [-stealth] (2) -- (3)node[ midway,sloped,below,font=\scriptsize]{$d_1$};
\draw [-stealth] (1) -- (4)node[ midway,sloped,above,font=\scriptsize]{$d_5$};
\end{scope}
\begin{scope}[yshift= -2.8cm, xshift=-3cm]
\node [fill=black,circle, inner sep=1pt,label=90:$x_1$] (1) at ( 0, -0.2) {};
\node [fill=black,circle, inner sep=1pt,label=225:$x_2$] (2) at (-1.2 ,-.1) {};
\node [fill=black,circle, inner sep=1pt,label=180:$x_3$] (3) at (.0 ,-2) {};
\node [fill=black,circle, inner sep=1pt,label=45:$x_4$] (4) at (1.2 ,.1) {};

\draw [-stealth ] (1) -- (2) node[ midway,sloped,above,font=\scriptsize]{$d_1$};
\draw [-stealth] (3) -- (1)  node[ midway,sloped,above,font=\scriptsize]{$d_3+\mu_1$}; ;
\draw [-stealth] (4) -- (3)  node[ midway,sloped,below,font=\scriptsize]{$d_4$};
\draw [-stealth] (2) -- (3)  node[ midway,sloped,below,font=\scriptsize]{$d_2$};
\draw [-stealth] (1) -- (4)  node[ midway,sloped,above,font=\scriptsize]{$d_5$};
\end{scope}
\end{tikzpicture}

\caption{\small We illustrate the stability properties of ancillary and design equilibria around $\mathcal{S}_0$. Let the vector $(d_1,d_2,d_3,d_4,d_5) \in \mathcal{S}_0$. The horizontal dashed line corresponds ancillary equilibria and the slanted line that intersects it to design equilibria.   They coincide at $\mu = 0$, as  seen in Proposition~\ref{prop:allequ};   for $\mu \neq 0$ configurations in $\mathcal{S}_0$ are ancillary equilibria.   For $\mu_1 <0$, there is an ancillary equilibrium  with $e_2,e_3,e_4 = 0$ but $e_1=\varepsilon_1$ and $e_5 =\varepsilon_2$ and $z_1$ and $z_5$ aligned. It is illustrated in the top-left corner of the figure. This equilibrium is moreover stable. For $\mu_2 >0$, there is a similar ancillary equilibria with $e_1=\varepsilon_3$ and $e_5 = \varepsilon_4$, illustrated in the bottom-right corner, but this equilibrium is unstable. We see that around the bifurcation value $\mu_0$, there is a {\it transfer of stability} from $\mathcal{E}_d$ to $\mathcal{E}_a$. The orientation may be reversed (i.e.  $\mu_1 > 0, \mu_2 <0$ and all else the same in the figure) depending on the sign of the second derivatives in Theorem~\ref{th:tech1}. }
\label{fig:log4eq}
\end{center}

\end{figure}
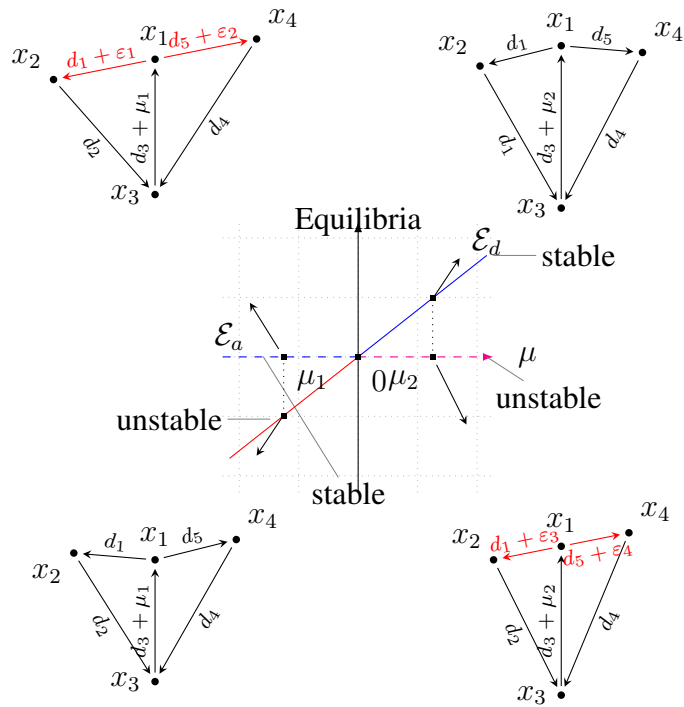
We illustrate the situation in Figure~\ref{fig:log4eq}. We will show that there is a set of positive measure  in $\mathcal{L}$ which cannot be made robustly almost surely globally stable. We do so by showing that for any framework attached to distances in that set, the existence of a stable ancillary equilibrium is generic for $\mathcal F$.

Denote by $\mathcal{S}^\varepsilon$ a tubular neighborhood of $\mathcal{S}$:  $$\mathcal{S}^\varepsilon = \lbrace d \in \mathcal{L} \mbox{ s.t. } \exists~	 d_0 \in \mathcal{S} \mbox{ with } \|d-d_0\| < \varepsilon\rbrace$$ and   $\mathcal{S}_0^\varepsilon= \mathcal{S}^\varepsilon \cap \mathcal{S}_0$. The set $\mathcal{S}^\varepsilon$ contains frameworks where $z_1$ and $z_5$ are close to parallel. These frameworks are infinitesimally rigid and non-singular.  Let $d \in \mathcal{S}_0^\varepsilon$ and $d_0 \in \mathcal{S}_0$ be such that there is $-\varepsilon<\mu<\varepsilon$ with $d=d_0+(0,0,\mu,0,0)$. Such  $d_0$ and $\mu$  exist by definition of $\mathcal{S}_0^\varepsilon$.

Because the system is invariant under mirror symmetry~\cite{belabbasSICOpart1}, the stability properties of the equilibria (a) and (c), (b) and (d) in Figure~\ref{fig:4formations} are the same. Assume without loss of generality that $u$ is such that the design equilibria for the frameworks with  $x_2$ and $x_4$ on different side of $z_3$ are stable. Because for a robust $u$, the system undergoes a transcritical bifurcation when $\mu=0$ by Theorem~\ref{th:tech1}, and because $u' \neq 0$ generically when $u$ vanishes, we have that for $\varepsilon$ small enough,  $\mathcal{E}_a$ contains the framework where $z_1$ is parallel to $z_5$ for all frameworks with $-\varepsilon<\mu<\varepsilon$. Furthermore, for either $\mu>0$ or $\mu<0$, we have that this framework is asymptotically stable, i.e. $\mathcal{E}_s \cap \mathcal{E}_a \neq \emptyset$. Hence, there is a set of positive measure of target frameworks in $\mathcal{S}_0^\varepsilon$ which contains a stable ancillary equilibrium and thus the system is not robustly almost-sure globally stable.

\end{proof}

We would like to thank Prof. B.D.O Anderson, Prof. Brockett, Prof. S. Morse as well as Alan O'Connor for helpful discussions. We are particularly grateful to Prof. Morse for introducing us to this problem.

\bibliographystyle{IEEEtran}
\bibliography{distribcontrolbib2}             

\appendix

The main tool handling  genericity and robustness in function spaces is Thom's transversality theorem. We will arrive at the result by building onto the simpler concept of transversality of linear subspaces.

\begin{figure}[ht]
\begin{center}
\begin{tikzpicture}[domain=-3:3,scale=.6] 
\draw[help lines, dotted,color=gray] (-3.1,-1) grid (3.1,5.1);
\draw[->] (-3.2,0) -- (3.2,0) node[right] {$x$}; 
\draw[->] (0,-1.1) -- (0,3.2) node[above] {$y=u(x)$};
\draw[color=black]	plot (\x,1/3*\x*\x)	node[right] {$u(x)$}; 
\draw[color=blue,dashed]	plot (\x,1/3*\x*\x+1/3*\x-.3)	node[right] {$\tilde u(x) $}; 
\draw[color=red,densely dotted] plot (\x,{1/2.9*\x*\x+1/3*\x+.5}) node[right] {$\tilde{\tilde{u}}(x)$};
\end{tikzpicture}
\caption{\small We let $\mathcal P$ be the property of vanishing with a zero derivative. Then $1- \mathcal P$ is generic and thus  $\mathcal P$  is not robust. Let $u(x)$ be a function which satisfy $\mathcal P$. For almost all perturbations, it will either vanish with a non-zero derivative---as illustrated with $\tilde u(x)$, dashed curve--- or not vanish at all---as illustrated with $\tilde{\tilde{u}}(x) $, dotted curve. Both $\tilde u(x)$ and $\tilde{\tilde{u}}(x)$ are transversal to the manifold defined by $y=0$ everywhere, whereas $u(x)$ is not.}\label{fig:illpertxsq}
\end{center}
\vspace{-.5cm}
\end{figure}
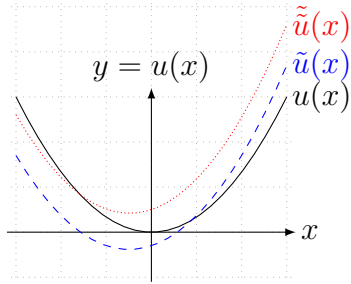

Let  $A, B \subset \R^n$ be linear subspaces. They are \emph{transversal} if $$\R^n = A \oplus B, $$ where $\oplus$ denotes the direct sum. For example, a plane and a line not contained in the plane are transversal in $\R^3$. The notion of transversality can be extended to maps as follows: given $$F_1 : \R^n \rightarrow \R^m\mbox{ and } F_2: \R^l \rightarrow \R^m,$$ we say that $F_1$ and $F_2$ are transversal at a point $(x_1,x_2) \in \R^n \times \R^l$ if one of the two following conditions is met:

\begin{enumerate}
\item $F_1(x_1) \neq F_2(x_2)$
\item If $F_1(x_1) = F_2(x_2)$, then the matrix $\left[ \begin{matrix}\frac{\partial F_1}{\partial x} & \frac{\partial F_2}{\partial x} \end{matrix}\right]^T$ is of full rank.
\end{enumerate}
In particular, if $l+n < m$ then $F_1$ and $F_2$ are transversal only where they do not map to the same point.

The notion of transversality that is of interest to us is a straightforward extension of the transversality of maps:

\begin{Definition}[Transversality]\label{def:transversality} Let $F:M \rightarrow N$ be a smooth map and let $C$ be a submanifold of $N$. Then $F$ is \emph{transversal} to $C$ at a given point if, at that point,  $F$ is transversal to the embedding $i:C \rightarrow N$ of $C$ into $N$. 
\end{Definition}
The definition is best understood from the following example:

\begin{Example}
Take $N=\R^3$ with coordinates $u,v,w$ and $C$ be the u-v plane. Let $F:\R \rightarrow \R^3:x \rightarrow [x,2x,3x]^T$. Then the map $F$ is transversal to $C$ everywhere since $F(x) \notin C$ for $x \neq 0$ and at $x=0$, $\frac{\partial F}{\partial x}$ is not in the tangent space of $C$.
\end{Example}

Let $F,G :M \rightarrow N$ be  smooth maps between smooth manifolds $M$ and $N$ . We say that $F$ and $G$ are 0-equivalent at $x_0$ if $F(x_0)=G(x_0),$  1-equivalent if in addition to being 0-equivalent,  $\frac{\partial F}{\partial x}|_{x_0}=\frac{\partial G}{\partial x}|_{x_0},$ and so forth. We define the k-jet of a smooth map to be its k-equivalence class:

\begin{Definition}  
The k-jet of $F:M \rightarrow N$ at $x_0$ is $$J_{x_0}^k(F) = \lbrace G:M \rightarrow N\mbox{ s.t. } G \mbox{ is }\mbox{k-equivalent to } F \rbrace .$$
\end{Definition}
Hence, the 0-jet of $F$ at $x_0$ is $F(x_0)$; the 1-jet is $(F(x_0), \frac{\partial F}{\partial x}|_{x_0})$, etc.  For example, the constant function $0$ and $\sin(x)$ have the same 0-jet at $x=0$ and $x$ and $\sin(x)$ have the same 1-jet at $0$. 

We define: $$J^k(M,N) = \mbox{ Space of k- jets from }M\mbox{ to } N .$$

Given a function $F:M \rightarrow N$, we call its \emph{k-jet extension} the map given by
$j^k_F(x):M \rightarrow J^K(M,N): x \rightarrow (F(x), \frac{\partial F}{\partial x}(x),\ldots, \frac{\partial^k F}{\partial x^k}(x)).$

\begin{Example} Let $M=N=\R$. The k-jet space is  $J^k(\R,\R)= \R \times \R \times \ldots \times \R = \R^{k+2}$. Take $F(x)=\sin(x)$; the $2-$jet extension of $F$ is $$j^2_{\sin}(x) = (x,\sin(x), \cos(x),-\sin(x)).$$
If we take $M=N=\R^2$ and $F(x)=Ax$ for $A \in \R^{2 \times 2}$, then $j_{Ax}^k(x)=(x,Ax, A, 0, \ldots, 0)$.
\end{Example}

While to any function $F:M \rightarrow N$, we can assign a k-jet extension $j_F^k:M \rightarrow J^k(M,N)$, the inverse is not true: there are maps $G:M \rightarrow J^k(M,N)$ which do not correspond to functions from $M$ to $N$ as there are some obvious integrability conditions that need to be satisfied. For example, if we let $$G:\R^n \rightarrow J^1(\R^n): G(x)=(x,Ax,B),$$ then  $G$ is a 1-jet extension of a function if and only if $B = A$.

The power of the transversality theorem of Thom is that it allows one  to draw conclusions about transversality properties in general, and genericity in particular, by \emph{solely looking at perturbations in jet spaces that are jet extensions}---a much smaller set than  all perturbations in jet-spaces.

We recall that the $C^r$ topology is a metric topology. It is induced by a metric that takes into account the function and its first $r$ derivatives. We have:

\begin{Theorem}[Thom's transversality]\label{th:thom} Let $C$ be a  regular submanifold of the jet space
$J^k(M,N)$. Then the set of maps $f : M  \rightarrow  N$ whose k-jet extensions are transversal to $C$
is an everywhere dense intersection of open sets in the space of smooth maps for the $C^r$ topology, $1 \leq r \leq \infty$.
\end{Theorem}

A typical application of the theorem is to prove that vector fields with degenerate zeros are not generic. We here prove a version of this result that is of interest to us. 

\begin{Corollary}\label{cor:corgent}
Functions in $\mathcal{C}^\infty(M) $ whose derivative at a zero vanish are not generic.
\end{Corollary}

In other words, the corollary deals with the intuitive fact that if $u(x)=0$, then generically $u'(x)\neq 0$. 
\begin{proof}
 Consider the space of 0-jets $J^0(M,\R)$. In this space, let $C$ be the set of 0-jets which vanish, i.e. $C=(x,0) \subset J^0$. A function $u$ is transversal to this set if either it does not vanish, or where it vanishes we have that  the matrix $$\left[ \begin{matrix} 1 & 1 \\ 0 & \frac{\partial f}{\partial x}\end{matrix}\right] $$ is of full rank. Hence, transversality to $C$ at a zero implies that the derivative of the function is non-zero. The result is thus a consequence of Theorem~\ref{th:thom}. 
\end{proof}


\end{document}